\documentclass{article}
\usepackage[margin=3cm]{geometry}

\usepackage{amsmath,amssymb,amsfonts,amsthm}
\usepackage{graphicx}
\usepackage{mathrsfs}
\usepackage{upgreek}
\usepackage{stackrel}
\usepackage{mathpazo}

\usepackage{xcolor}
\usepackage{hyperref}
\usepackage[nameinlink,noabbrev]{cleveref} 
\usepackage{algorithm}
\usepackage{algpseudocode}
\usepackage{enumerate}
\usepackage{bbm}
\usepackage{verbatim}
\usepackage{subcaption}
\usepackage{accents}

\usepackage[utf8]{inputenc}

\DeclareUnicodeCharacter{2080}{\textsubscript{0}}
\DeclareUnicodeCharacter{2081}{\textsubscript{1}}
\DeclareUnicodeCharacter{2082}{\textsubscript{2}}
\DeclareUnicodeCharacter{2083}{\textsubscript{3}}
\DeclareUnicodeCharacter{2084}{\textsubscript{4}}
\DeclareUnicodeCharacter{2085}{\textsubscript{5}}
\DeclareUnicodeCharacter{2086}{\textsubscript{6}}
\DeclareUnicodeCharacter{2087}{\textsubscript{7}}
\DeclareUnicodeCharacter{2088}{\textsubscript{8}}
\DeclareUnicodeCharacter{2089}{\textsubscript{9}}
\DeclareUnicodeCharacter{207D}{\textsuperscript{(}}
\DeclareUnicodeCharacter{207E}{\textsuperscript{)}}

\newtheorem{theorem}{Theorem}[section]
\newtheorem{proposition}[theorem]{Proposition}
\newtheorem{lemma}[theorem]{Lemma}
\newtheorem{corollary}[theorem]{Corollary}

\theoremstyle{definition}
\newtheorem{definition}[theorem]{Definition}
\newtheorem{remark}[theorem]{Remark}
\newtheorem{example}[theorem]{Example}

\usepackage{titlesec}

\titleformat{\section}
  {\normalfont\large\bfseries}
  {\thesection}
  {0.5em}
  {}

\titleformat{\subsection}
  {\normalfont\normalsize\bfseries}{\thesubsection}{0.5em}{}

\newcommand{\comp}{\mathrm{Comp}}

\title{Signature Varieties of Splines}
\date{}
\author{Carlos Améndola\thanks{TU Berlin, Berlin, Germany} \and Felix Lotter\thanks{MPI MiS, Leipzig, Germany} \and Leonard Schmitz\footnotemark[1]}

\begin{document}

\maketitle

\begin{abstract}
    Splines are central objects for the interpolation of discrete data via piecewise smooth paths. Their iterated-integral signature is an infinite collection of tensors which characterizes paths almost uniquely. We study truncations of this collection, which define algebraic maps from parameter space to tensor space.
    
    We prove that the images of these maps are given by orbits of a matrix-tensor action. Furthermore, taking the Zariski closure, we define and study varieties of spline signature tensors. We determine dimension and degree of these tensor varieties in a number of examples, relying on symbolic computations.

    With a view towards learning, constructing paths with a given signature tensor translates to studying the fibers of the signature map. We use computational methods to determine their cardinality, with a focus on its dependence on different classes of splines. We observe in explicit examples that reconstructing splines from a given signature tensor of a path yields close approximations of the original path.
\end{abstract}

\section{Introduction}

Path signatures were introduced by Chen \cite{chen1954iterated}, playing a key role in rough path theory \cite{lyons2007differential,friz2010multidimensional}. Their applicability in different fields has only grown since then, including finance \cite{Cuchiero2026}, machine learning \cite{Chevyrev2026} and topological data analysis \cite{chevyrev2018persistence}. A concrete connection to algebraic geometry was initiated in \cite{AFS19}, where signature varieties associated to special families of paths were studied; with major interest in piecewise linear and polynomial paths. This viewpoint proved particularly useful when studying the problem of learning paths from their signature tensors \cite{pfeffer2019learning}. 

In this work we generalize the two families above by considering piecewise polynomial paths, as well as smooth versions giving rise to \emph{splines}. Parametric splines are central objects for the interpolation of discrete data, whose applications range from computer-aided geometric design and computer graphics to signal processing, numerical simulation, and data approximation in engineering and physics \cite{ahlberg2016theory,micula2012handbook}.  

We study varieties of signature tensors of piecewise polynomial paths, constructing \textit{dictionaries} and \textit{core tensors} in the sense of \cite{pfeffer2019learning}. We also study dimensions, degrees and equations for the resulting varieties. We perform explicit symbolic computations in \textsc{Macaulay2} \cite{M2} and \texttt{OSCAR} \cite{OSCAR,OSCAR-book}, and numerical computations using \texttt{HomotopyContinuation.jl} \cite{BT18}. In particular, we use the signature packages \cite{amendola2025computing,RS26}. 

\subsection*{Acknowledgements} 
The authors acknowledge funding by the Deutsche Forschungsgemeinschaft (DFG, German Research Foundation)
– CRC/TRR 388 “Rough Analysis, Stochastic Dynamics and Related Fields” – Project A04, 516748464.
\section{Splines and their signature tensors}

All splines are special cases of piecewise polynomial paths. Throughout the paper, we assume that piecewise polynomial paths are given as follows:

\begin{definition}\label{def:pwpolypath} A
   \textit{piecewise polynomial path} is a continuous map $X: [0,1]\to \mathbb R^d$ such that for a suitable $\ell$ and all $0\leq i \leq \ell-1$, the restriction $X|_{[\frac i \ell,\frac {i+1}\ell]}$ is polynomial. We call the points $(\frac i \ell, X(\frac i\ell))$ of the graph of $X$ the \textit{control points}. 
\end{definition}

Note that we restrict the control points of a piecewise polynomial path in this paper to be equidistant in time. We will see later that depending on the type of regularity we impose at the control points, this is possible without loss of generality. An advantage of our assumption is that each piecewise polynomial path can simply be written as a concatenation 
$$X=X^{[1]} \sqcup \ldots \sqcup X^{[\ell]}$$ of polynomial paths $X^{[i]}: [0,\frac1\ell
]\to \mathbb R^d$ defined on the same interval. Recall that the \emph{path concatenation} of $Y: [0,t_1] \to \mathbb R^d$ and $Z: [0,t_2] \to \mathbb R^d$ is the unique map denoted by $Y\sqcup Z:[0,t_1+t_2]\to \mathbb R^d$ that agrees with $Y$ on $[0,t_1]$,  and with $Y(t_1)-Z(0) + Z$ on $[t_1,t_1+t_2]$.

To introduce the signature of a path, we fix a truncation level $k$,  and
consider the truncated tensor algebra $$T^{\leq k}(\mathbb R^d) := \bigoplus_{j=0}^k (\mathbb R^d)^{\otimes j}=\mathbb R\oplus\mathbb R^d\oplus(\mathbb R^d)^{\otimes 2} \oplus\dots\oplus(\mathbb R^{d})^{\otimes k}$$
serving as our $(d^{k+1}-1)/(d-1)$-dimensional ambient space. This space is a direct sum of $j$-tensors on $\mathbb R^d$.
\begin{definition}
    For a piecewise polynomial path $X$, the $k$-truncated \emph{(iterated-integrals) signature} $\sigma^{\leq k}(X)$ is a direct sum
    $$\sigma^{\leq k}(X):=\sigma^{(0)}(X)\oplus\dots\oplus\sigma^{(k)}(X)\in T^{\leq k}(\mathbb R^d)$$
    of \emph{signature tensors} $\sigma^{(j)}(X)\in(\mathbb R^d)^{\otimes j}$ with entries
\begin{align*}
  \label{eq:integrals}
  (\sigma^{(j)}(X))_{i_1\dots i_j}:=\int_0^1\int_0^{t_j}\dots\int_0^{t_2} \dot X_{i_1}(t_1) \dots \dot X_{i_j}(t_j)\,\mathrm dt_1 \dots \mathrm dt_j
\end{align*}
corresponding to
$1\leq i_1,\dots ,i_j\leq d$. 
Here, $\dot X$ denotes the differential of $X$, and we set $\sigma^{(0)}(X):=1$. 
\end{definition}
The signature $\sigma^{\leq k}(X)$ lies in the \textit{free $k$-step nilpotent Lie group} $$\mathcal G^{\leq k}(\mathbb R^2)\subseteq T^{\leq k}(\mathbb R^d),$$ see, for instance, {\cite[Theorem 7.30]{friz2010multidimensional}}. The group multiplication is induced by the tensor product $\otimes$ and compatible with path concatenation, i.e., 
\begin{equation}\label{eq:chen}
    \sigma^{\leq k}(X\sqcup Y)=\sigma^{\leq k} (X)\cdot \sigma^{\leq k}(Y)
\end{equation}
for any $X$ and $Y$. This is known as \emph{Chen's identity} \cite[Theorem 7.11]{friz2010multidimensional}. 

Chen's identity implies the following useful property of piecewise polynomial paths:

\begin{proposition}\label{prop:SignatureIsMorphism}
    The signature of any piecewise polynomial path $X=X^{[1]} \sqcup \ldots \sqcup X^{[\ell]}$ is algebraic in the coefficients of the polynomials defining $X^{[1]},\ldots,X^{[\ell]}$.
\end{proposition}

We will give a proof in \Cref{sec:dictionarys}.\\

A general piecewise polynomial path can be viewed as a \textit{spline of regularity} $0$. To define splines in higher regularity, we need to introduce smoothness constraints at the control points. What these smoothness constraints should be depends on the context. We will distinguish two kinds of regularity, with terminology inspired by \cite[p. 63-64]{BarskyParametric1989}.

\begin{definition}[Geometric and parametric splines]\label{def:splines}
    Let $$X^{[1]},\ldots,X^{[\ell]}: [0,\tfrac{1}{\ell}] \to \mathbb R^d$$ be polynomial paths. For $r \in \mathbb N$, we call the piecewise polynomial path $X = X^{[1]} \sqcup \ldots \sqcup X^{[\ell]}$ a \textit{geometric spline} of regularity $r$ if, for all $1 \leq i \leq \ell-1$ and $1 \leq s \leq r$,
    \begin{equation}\label{eq:geom-cont}
        \rho_{i,s} \frac{\mathrm dX^{[i]}}{\mathrm dt^s} \left(\frac{1}{\ell}\right) = \frac{\mathrm dX^{[i+1]}}{\mathrm dt^s}(0)
    \end{equation}
    for some $\rho_{i,s} > 0$. We call it a \textit{parametric spline} of regularity $r$ if \eqref{eq:geom-cont} holds with $\rho_{i,s} = 1$ for all $i$ and $s$.
\end{definition}

Note that for $r = 0$, condition \eqref{eq:geom-cont} is empty.

\begin{figure}
    \centering
    \includegraphics[width=0.6\linewidth]{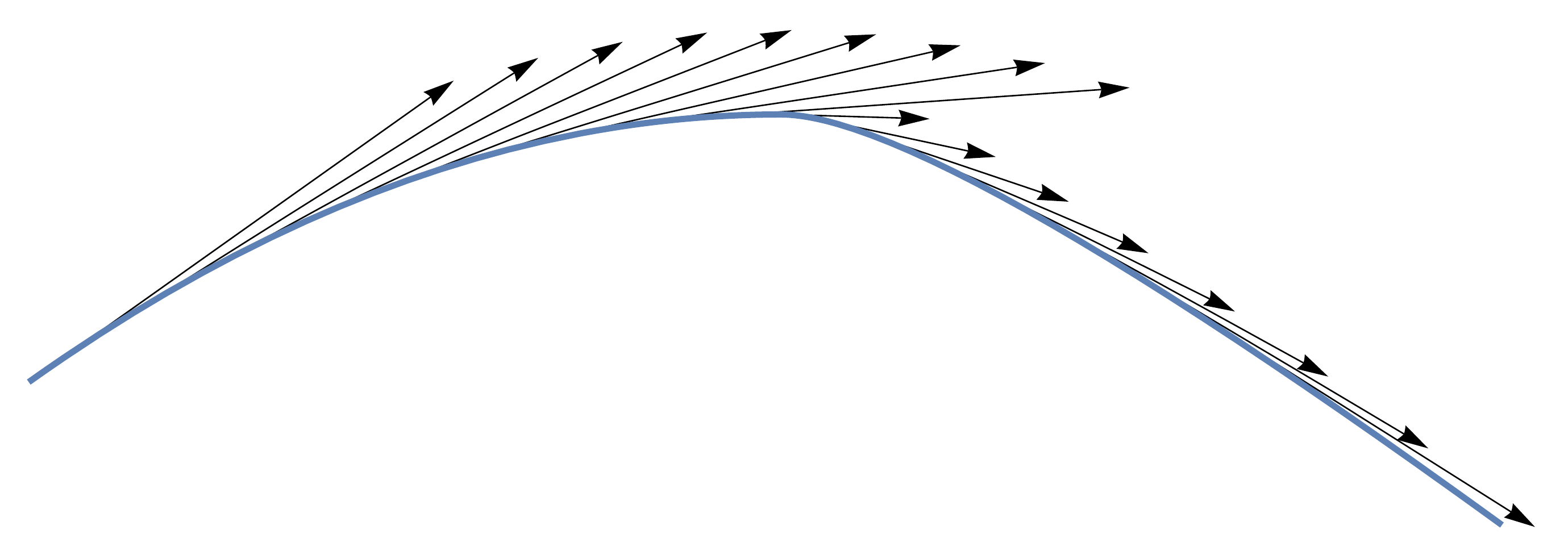}
    \caption{A geometric $(2,2)$-spline of regularity $1$}\label{fig:geom-spline}
\end{figure}

\begin{remark}
    Parametric splines of regularity $r$ are given by those geometric splines whose coordinates are $\mathcal C^r$-functions. This notion puts a lot of emphasis on the parametrization. \Cref{fig:geom-spline} shows a geometric spline which is not parametric: at the control point, the tangent vectors are parallel, but their magnitude is discontinuous.
\end{remark}

\begin{example} The path 
     $(2t,4t) \sqcup (4t,8t)$, with both sub-paths defined on $[0,\frac 1 2]$, is not a parametric spline, even though it is just a reparametrization of the linear path $(3t,6t)$ defined on $[0,1]$. Note that both paths are geometric splines of regularity $1$.
\end{example}

Denote by $\comp_{\geq r}(M)$ the set of compositions of $M\in\mathbb{N}$ with each entry greater or equal to $r$, i.e., all $m=(m_1,\dots,m_\ell)$ with $m_i\geq r$ such that $m_1+\dots+m_\ell=M$. We call $\ell$ the \emph{length} of the composition $m$. If $r=1$ we also write $\comp(M)$. 

\par 

For $m\in\comp_{\geq r}(M)$ we call $X$ a \emph{geometric/parametric $m$-spline of regularity $r$}, if $X$ is a geometric/parametric spline with regularity $r$, and for each $i$ the polynomials defining $X^{[i]}$ in \Cref{def:splines} have degree less or equal $m_i$. 

\begin{example}
    The path $(t,t^2) \sqcup (-t,-t)$ is a geometric (and parametric) $(2,1)$-spline of regularity $0$, but not of any higher regularity. The path $(t,t^2) \sqcup (t, 2t - t^2)$ is a parametric $(2,2)$-spline of regularity $1$. The path $(2t,4t) \sqcup (4t,8t)$ is a geometric $(1,1)$-spline of regularity $1$. Its reparametrization $(3t,6t)$ is a geometric (and parametric) $(1)$-spline of every regularity.
\end{example}

\section{Signature varieties}

Recall that the truncated signature $\sigma^{\leq k}(X)$ of any path $X$ is an element of the free $k$-step nilpotent Lie group $\mathcal G^{\leq k}(\mathbb R^d)$. This is an affine algebraic subvariety of $T^{\leq k}(\mathbb R^d)$, cut out by \textit{shuffle relations} \cite[Theorem 4.10]{AFS19}. \par
By the Chen-Chow theorem \cite[Theorem 7.28]{friz2010multidimensional}, every element of $\mathcal G^{\leq k}(\mathbb R^d)$ can be obtained as the signature $\sigma^{\leq k}(X)$ of some piecewise linear path $X$. In particular, there are no algebraic relations on signatures of general piecewise polynomial paths. However, restricting to subclasses of parametric or geometric splines, we obtain interesting subvarieties of $\mathcal G^{\leq k}(\mathbb R^d)$.

\begin{definition}
    Let $d,k,M\in \mathbb N$,  $r \in \mathbb N$ and $m\in\comp_{\geq r}(M)$. The \textit{signature variety of geometric splines} associated to $d,k,m$ and $r$ is the Zariski closure (over $\mathbb C$), 
    $$\mathcal{S}_{d,\leq k,m}^r := \overline{\{\sigma^{\leq k}(X) \ | \ X \in \mathcal X\}} \subseteq T^{\leq k}(\mathbb C^d),$$
    where $\mathcal X$ is the set of geometric $m$-splines of regularity $r$ in $\mathbb R^d$. Similarly, we define the \textit{signature variety of parametric splines}
    $$\mathcal{P}_{d,\leq k,m}^r \subseteq \mathcal{S}_{d,\leq k,m}^r$$
    by restricting $\mathcal X$ to be the subset of parametric $m$-splines of regularity $r$.
\end{definition}

Note that we define all varieties over $\mathbb C$, as usual in applied algebraic geometry. We will see in \Cref{cor:coretensorforParametricSplines} that the set $\{\sigma^{\leq k}(X) \ | \ X \in \mathcal X\}$ is the image of a polynomial map between real affine spaces; in particular it is a semi-algebraic set by the Tarski-Seidenberg theorem. One should think of signature varieties as algebraic approximations of this set.

\begin{remark}\label{rem:lyndon-coordinates}
Depending on the context, we might want to view $\mathcal S^r_{d,\leq k,m}$ and $\mathcal P^r_{d,\leq k,m}$ as embedded into either $\mathcal G^{\leq k}(\mathbb C^d)$ or $T^{\leq k}(\mathbb C^d)$. The variety $\mathcal G^{\leq k}(\mathbb C^d)$ can in fact be identified with an affine space as well, by choosing free coordinates as follows:

A (nonempty) string over an ordered alphabet is called a \emph{Lyndon word} if it is strictly smaller in lexicographic order than all of its right subwords. For example, $\mathtt{122}$ is a Lyndon word over $\{\mathtt{1},\mathtt{2}\}$, while $\mathtt{212}$ is not, as $\mathtt{212}>\mathtt{12}$. Now for $x \in T^{\leq k}(\mathbb C^d)$ and a Lyndon word $w$ of length $\leq k$ over $\{\mathtt 1,\ldots,\mathtt d\}$, let $x_w$ denote the associated coordinate of $x$. Then the $x_w$ define free coordinates on $T^{\leq k}(\mathbb C^d)$; cf.\ e.g.\ \cite[Theorem 4.10]{AFS19}.
We will take $\mathbb C[x_w \ | \ w \text{ Lyndon of length} \leq k]$ to be the coordinate ring of $\mathcal G^{\leq k}(\mathbb C^d)$, where we grade each generator by the length of the associated word.
\end{remark}

\begin{example}
    The variety $\mathcal P^1_{2,\leq 3,(2,1)}$ is a hypersurface in the $5$-dimensional ambient space $\mathcal G^{\leq 3}(\mathbb R^2)$. Its equation is
    $$960x_{112}x_2 + 960x_{122}x_1 - 612x_{12}^2 - 348 x_{12} x_1 x_2 + 7 x_1^2 x_2^2 = 0.$$
    As an element of the tensor algebra, this relation takes the form 
    $$28 (\texttt{2211} + \texttt{1122}) + 292 (\texttt{1221} + \texttt{2112}) - 320 (\texttt{1212} + \texttt{2121}).$$
\end{example}

\begin{proposition}
    The ideals of $\mathcal{S}_{d,\leq k,m}^r$ and $\mathcal{P}_{d,\leq k,m}^r$ in the graded coordinate ring of $\mathcal G^{\leq k}(\mathbb C^d)$ are homogeneous.
\end{proposition}
\begin{proof}
    Let $\mathcal X$ be the set of parametric $m$-splines of regularity $r$ in $\mathbb R^d$ and $I$ the ideal of $\mathcal S_{d,\leq k,m}^r$ in $\mathbb C[x_w \ | \ w \text{ Lyndon of length} \leq k]$. If $X \in \mathcal X$, so is the scaled path $\lambda X$ for every $\lambda \in \mathbb R$. By equivariance, $\sigma_w(\lambda X) = \lambda^{\ell(w)}\sigma_w(X)$, where $\ell(w)$ is the length of $w$. Now take $f \in I$ and write $f = f_1 + \ldots + f_k$ as a sum of its homogeneous parts. Let $S:= \sigma^{\leq k}(X)$ for a path $X \in \mathcal X$. The above implies that $\sum_i \lambda^i f_i(S) = 0$ for all $\lambda$. But this forces $f_i(S) = 0$ and thus $f_i \in I$ for all $i$. The same argument applies for geometric $m$-splines.
\end{proof}

\begin{remark}
    By taking the Zariski image of the varieties $\mathcal S^r_{d,\leq k,m}$ under the projection $$T^{\leq k}(\mathbb C^d)\to (\mathbb C^d)^{\otimes k},$$
we obtain varieties $\mathcal S^r_{d,k,m}\subseteq\mathcal P^r_{d,k,m}$ of fixed level signature tensors. In particular, the case $r=0$ includes the signature varieties of piecewise linear and polynomial paths that were introduced in \cite[Equation (37)]{AFS19}. Since the projection above is a finite map, it preserves dimension, see \cite[Theorem 6.1]{AFS19}. The image of the ambient space $\mathcal G^{\leq k}(\mathbb C^d)$ under the projection is the \textit{universal variety} from \cite[Equation 30]{AFS19}.\par
Varieties of truncated signatures  (as opposed to fixed level signatures) were also used in path recovery tasks \cite[Section 6]{amendola2025learning}, to benefit from the Lie group structure defining barycenters.
\end{remark}

Signature varieties of splines define highly structured filtrations of the variety $\mathcal G^{\leq k}(\mathbb R^d)$. From the definition, for every $m \in \comp_{\geq r}(M)$ we have a grid of inclusions:
\[\begin{array}{ccccccccc}
    \mathcal P^r_{d,\leq k,m} & \subseteq & \mathcal P^{r-1}_{d,\leq k,m} & \subseteq & \ldots & \subseteq & \mathcal P^{1}_{d,\leq k,m} &\subseteq &\mathcal P^{0}_{d,\leq k,m} \\
    \rotatebox{-90}{$\subseteq$} & & \rotatebox{-90}{$\subseteq$} & & & & \rotatebox{-90}{$\subseteq$} & & \rotatebox{-90}{$\subseteq$} \\[1em]
    \mathcal S^r_{d,\leq k,m} & \subseteq & \mathcal S^{r-1}_{d,\leq k,m} & \subseteq & \ldots & \subseteq & \mathcal S^{1}_{d,\leq k,m} &\subseteq &\mathcal S^{0}_{d,\leq k,m}
\end{array}\]
Moreover, note that for vectors of natural numbers $m$ and $m'$ we have
\begin{equation}\label{eq:m-incl}
    \mathcal S^r_{d,\leq k,m} \subseteq \mathcal S^r_{d,\leq k,m'}
\end{equation}
if $\ell(m) \leq \ell(m')$ and there are $i_1 < \ldots < i_{\ell_m}$ such that $m_j \leq m'_{i_j}$ for every $1\leq j\leq \ell(m)$.

\begin{proposition}
    The varieties $\mathcal P^r_{d,\leq k,m}$ and $\mathcal S^r_{d,\leq k,m}$ are irreducible. For every $d,k$ and $r$ there is some $N\in \mathbb N$ such that $$\mathcal P^r_{d,\leq k,m} = \mathcal S^r_{d,\leq k,m} = \mathcal G^{\leq k}(\mathbb R^d)$$
    for all $m \in \comp_{\geq r + 1}(N)$ with $\sum_{i=1}^{\ell(m)} m_i \geq N$.
\end{proposition}
\begin{proof}
    The first statement follows since by \Cref{thm:geom-spline-dict} below both varieties can be described as the image of an irreducible parameter space under a polynomial map.\par
    Since for any $m$ we have $\mathcal P^0_{d,\leq k,(m_i)} \subseteq \mathcal P^r_{d,\leq k,m}$, the second statement certainly holds for all $m$ with length at most some $\ell \in \mathbb N$. By \cite[Theorem 5.6]{AFS19} there is some $N'$ such that $\mathcal P^0_{d,\leq k,(N')} = \mathcal G^{\leq k}(\mathbb R^d)$ and so we can choose $N = \ell \cdot N'$. It remains to show that there is some $L$ with $\mathcal P^r_{d,\leq k, m} = \mathcal G^{\leq k}(\mathbb R^d)$ once $\ell(m)\geq L$.
    
    We define a halfshuffle homomorphism (see e.g.\ \cite[Section 3.1]{colmenarejo2020signatures}) $\phi: T(\mathbb R^d) \to T(\mathbb R^d)$ by mapping a letter $\texttt i$ to $\texttt {id}^r$. We claim that this is an injection. Indeed, assume there is some linear combination $\sum_i \lambda_i w_i$ of words mapping to $0$. Now it is easy to see (e.g.\ by induction) that among the words appearing in $\phi(i_1\ldots i_n)$ there is a unique one where the letters $i_1,\ldots,i_n$ appear at the indices $zr$ for $z \in \mathbb N$, given by the concatenation of words $i_j\texttt d ^r$ for $1 \leq j \leq n$. This implies that already $\sum_i \lambda_i w_i = 0$, proving the claim.
    
    Since it is injective, $\phi$ induces a morphism $\mathcal G^{\leq k\cdot(r+1)}(\mathbb R^d) \to \mathcal G^{\leq k}(\mathbb R^d)$ whose image is Zariski dense. By Chen-Chow \cite[Theorem 7.28]{friz2010multidimensional}, there is some $L'$ such that for every $x \in \mathcal G^{\leq k\cdot(r+1)}(\mathbb R^d)$ there is a piecewise linear path $X$ with $L'$ segments and
    $$x = \sigma^{\leq k\cdot(r+1)}(X).$$
    By \cite[Theorem 3.9]{colmenarejo2020signatures} we have for any word $w$, 
    $$\sigma(X)_{\phi(w)} = \sigma(Y)_{w}$$
    where $Y$ is the path with $Y_i(s) := \sigma(X|_{[0,s]})_{i d\ldots d}$. We can assume by reparametrization that the $d$-th coordinate of $X$ is just given by the time variable $t$. Then $Y$ is a parametric spline of regularity $r$, with $L'$ pieces of degree $r+1$. This shows that the Zariski dense image of $\mathcal G^{\leq k\cdot(r+1)}(\mathbb R^d) \to \mathcal G^{\leq k}(\mathbb R^d)$ is contained in $\mathcal P^r_{d,\leq k,(r+1,\ldots,r+1)}$, so we can choose $L = L'$.
\end{proof}

\begin{proposition}\label{prop:antipode}
For all $m=(m_1,\dots,m_\ell)\in\comp_{\geq r}(M)$, 
\begin{enumerate}[(i)]
\item\label{prop:antipode1}
    $\mathcal{S}^r_{d,\leq k,(m_1,m_2,\dots,m_\ell)}\cong\mathcal{S}^r_{d,\leq k,(m_\ell,\dots,m_2,m_1)}$ and 
    \item\label{prop:antipode2} $\mathcal{P}^r_{d,\leq k,(m_1,\dots,m_\ell)}=\mathcal{P}^r_{d,\leq k,(m_1,\dots,m_i,m_{i+2},\dots,m_\ell)}$ if $m_i\!=\!m_{i+1}\!=\!r$. 
\end{enumerate}

\end{proposition}
\begin{proof}
 For every $m$-spline $X$ we have 
    $$\sigma^{\leq k}(X)_{i_1\dots i_k}=(-1)^k \sigma^{\leq k}(\accentset\leftharpoonup X)_{i_k\dots i_1}$$
    where $\accentset\leftharpoonup X$ is the $(m_\ell,\dots,m_1)$-spline that results from $X$ after time reversal; see \cite[Proposition 7.12]{friz2010multidimensional}. 
    Part \eqref{prop:antipode1} therefore follows by permuting the signature entries accordingly. 
    
    For part \eqref{prop:antipode2} we observe that the space of parametric $m$-splines coincides, up to reparametrization, with the space of $(m_1,\dots,m_i,m_{i+2},\dots,m_\ell)$-splines whenever $m_i=m_{i+1}=r$. Here we use that a piecewise polynomial path with (parametric) regularity $r$ and polynomial parts of degrees $r$ is already polynomial itself. The signature is invariant under reparametrization; see \cite[Proposition 7.10]{friz2010multidimensional}. Therefore, the images of parametric splines, and in particular their Zariski closures with respect to $m$ of length $\ell$ and $(m_1,\dots,m_i,m_{i+2},\dots,m_\ell)$ of length $\ell-1$ coincide. 
\end{proof}

\begin{remark}
    Property \eqref{prop:antipode2} in \Cref{prop:antipode} does not hold for geometric splines, cf.\ \Cref{tab:tensor-var-d3k3}: the variety $\mathcal S^2_{3,\leq 3,(2,2)}$ has dimension $7$ but $\mathcal S^2_{3,\leq 3,(2)} = \mathcal S^0_{3,\leq 3,(2)}$ is $6$-dimensional, see \cite[Table 3]{AFS19}.
\end{remark}

\section{Dictionaries and core tensors}\label{sec:dictionarys}

In \cite{pfeffer2019learning}, the notion of a \textit{dictionary} for a class of paths in $\mathbb R^d$ was introduced. This is a fixed path $\psi$ in a (possibly higher or lower-dimensional space) $\mathbb R^M$ such that any path in the class can be written as a linear transform of $\psi$, 
$$A \circ \psi:[0,1]\rightarrow\mathbb{R}^d,t\mapsto A \cdot \psi(t)$$ 
 where $A\in\mathbb R^{d\times M}$. The idea is that by equivariance of the signature, we can then express all signature tensors of paths in the class as images of the \textit{core tensors} $\sigma^{(k)}(\psi)$ under a matrix action, see \eqref{eq:equivariance_sig} below.\par
In this section, we prove that parametric $m$-splines of given regularity $r\geq 0$ admit a dictionary. More generally, we will see that geometric $m$-splines admit an \textit{algebraically parametrized} dictionary.\par
For a detailed introduction into dictionaries and core tensors we refer the reader to \cite[Section 2]{pfeffer2019learning}.

\begin{example}
    The \emph{moment curve} $\mathsf{Mom}:[0,1]\rightarrow \mathbb{R}^M$,  given by
    $$\mathsf{Mom}^M(t):=(t,t^2,\dots,t^{M})$$
    for all $t\in[0,1]$, is a dictionary for polynomial paths of degree $M$.
\end{example}

We recall the important notion of a \emph{congruence transform} on tensors.  
For a tensor $C\in(\mathbb{R}^M)^{\otimes k}$
 and matrix $A\in{\mathbb{R}}^{d\times M}$, the \emph{matrix-tensor congruence transform} $A*C\in(\mathbb{R}^{d})^{\otimes k}$ is a $k$-tensor with entries
\begin{equation}\label{eq:def_tensorMatrixCongruence}\left(A*C\right)_{i_1\dots i_k}:=\sum_{\alpha_1=1}^M\dots\sum_{\alpha_k=1}^MA_{i_1\alpha_1}\dots A_{i_k\alpha_k}C_{\alpha_1\dots\alpha_k}\end{equation}
for every $1\leq i_1,\dots,i_k\leq d$. Whenever $M$ is significantly smaller than $d$, the factorization $A* C$ is known as the \emph{Tucker format}. The matrix-tensor congruence defines an action on tensors. It naturally extends to sequences of tensors 
$$*:\mathbb{R}^{d\times M}\times T^{\leq k}(\mathbb R^M)\rightarrow T^{\leq k}(\mathbb R^d)$$
where we apply the action \eqref{eq:def_tensorMatrixCongruence} to every graded component simultaneously. 
Since integration is linear, it is easy to see (e.g. \cite[Lemma 2.1]{pfeffer2019learning}) that the signature is equivariant, i.e.,  
\begin{equation}\label{eq:equivariance_sig}
\sigma^{\leq k}(A\circ X)=A*\sigma^{\leq k}(X),
\end{equation}
for all $X:[0,1]\rightarrow\mathbb R^M$ and $A\in\mathbb R^{d\times M}$.\\

Let us now define a dictionary for the class of $m$-splines of regularity $0$. This yields a core tensor whose matrix congruence orbit will describe the varieties $\mathcal S^0_{d,\leq k,m}$ for arbitrary $d$. 

\begin{definition}\label{def:dictionaryPwPoly}
For every $m\in\comp(M)$ of length $\ell$ we define the \emph{piecewise monomial paths}
\begin{equation*}
\mathsf{PwMom}^{m}:=\begin{pmatrix}
\mathsf{Mom}^{m_1}\\
0_{m_{2}+\dots+m_{\ell}}
\end{pmatrix}\sqcup
\begin{pmatrix}0_{m_1}\\
\mathsf{Mom}^{m_2}\\
0_{m_{3}+\dots+m_{\ell}}
\end{pmatrix}
\sqcup\dots \sqcup 
\begin{pmatrix}0_{m_1+\dots+m_{\ell-1}}\\
\mathsf{Mom}^{m_\ell}
\end{pmatrix}
\end{equation*}
via path compositions of $\ell$ embedded moment curves 
\begin{equation}\label{eq:defEmbMom}
\begin{pmatrix}0_{m_1+\dots+m_{i-1}}\\
\mathsf{Mom}^{m_i}\\
0_{m_{i+1}+\dots+m_{\ell}}
\end{pmatrix}:[0,1]\rightarrow\mathbb{R}^M
\end{equation}
into $M$-dimensional space for $1\leq i\leq \ell$, reparametrized linearly to a path
$$\mathsf{PwMom}^{m}: [0,1] \to \mathbb R^{M}.$$
\end{definition}

In 
\Cref{fig:axis3_compositions} we illustrate several piecewise monomial paths according to \Cref{def:dictionaryPwPoly}. This refines the well-known dictionaries for piecewise linear paths $m=(1,1,1)$ and polynomial paths $m=(3)$ that are used in \cite{pfeffer2019learning}. 

\begin{figure}[h]
    \centering
    \begin{subfigure}{0.24\columnwidth}
        \centering
\includegraphics[width=\linewidth]{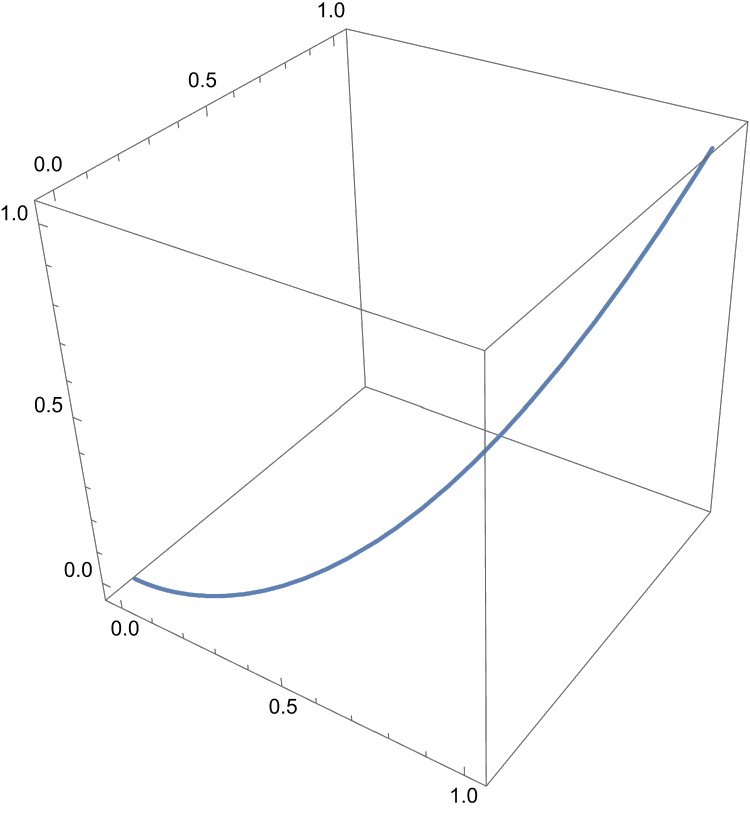}
\caption*{ 
$m=(3)$}
\end{subfigure}
\hfill
        \begin{subfigure}{0.24\columnwidth}
        \centering
\includegraphics[width=\linewidth]{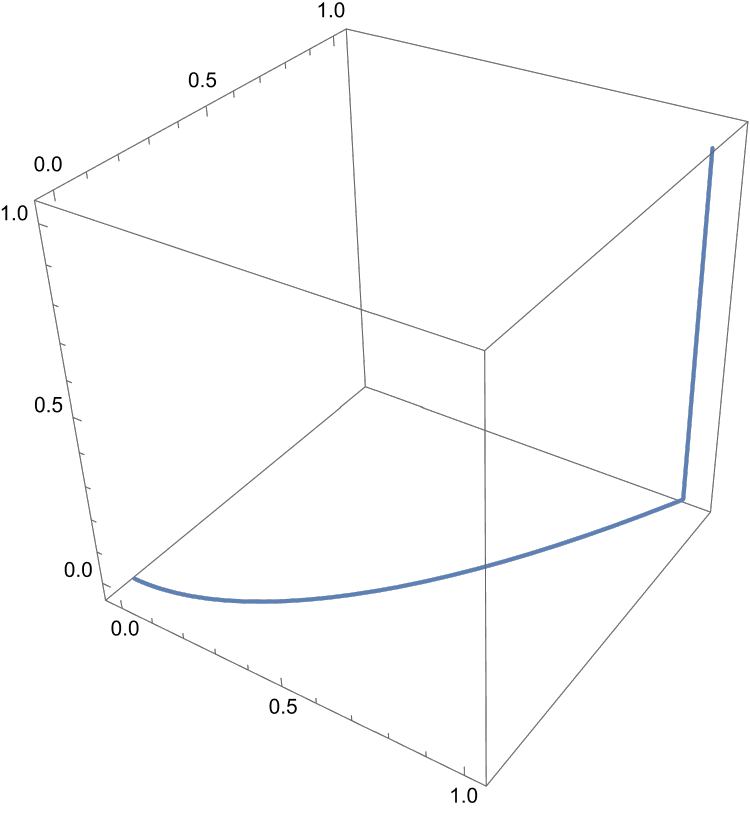}
\caption*{$m=(2,1)$}
\end{subfigure}
\hfill
\begin{subfigure}{0.24\columnwidth}
        \centering
\includegraphics[width=\linewidth]{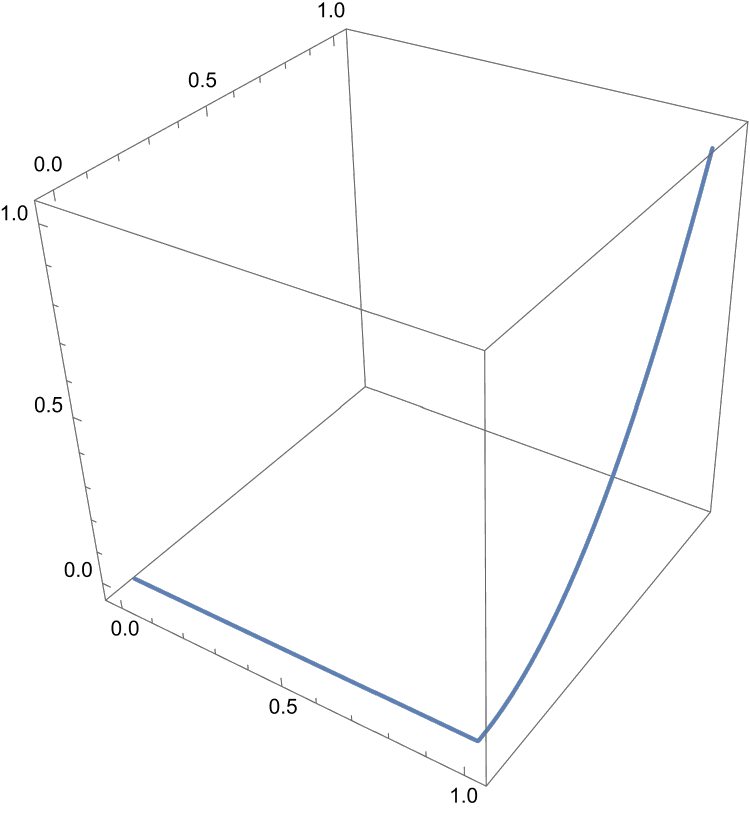}
\caption*{$m=(1,2)$}
\end{subfigure}
\hfill
\begin{subfigure}{0.24\columnwidth}
        \centering
\includegraphics[width=\linewidth]{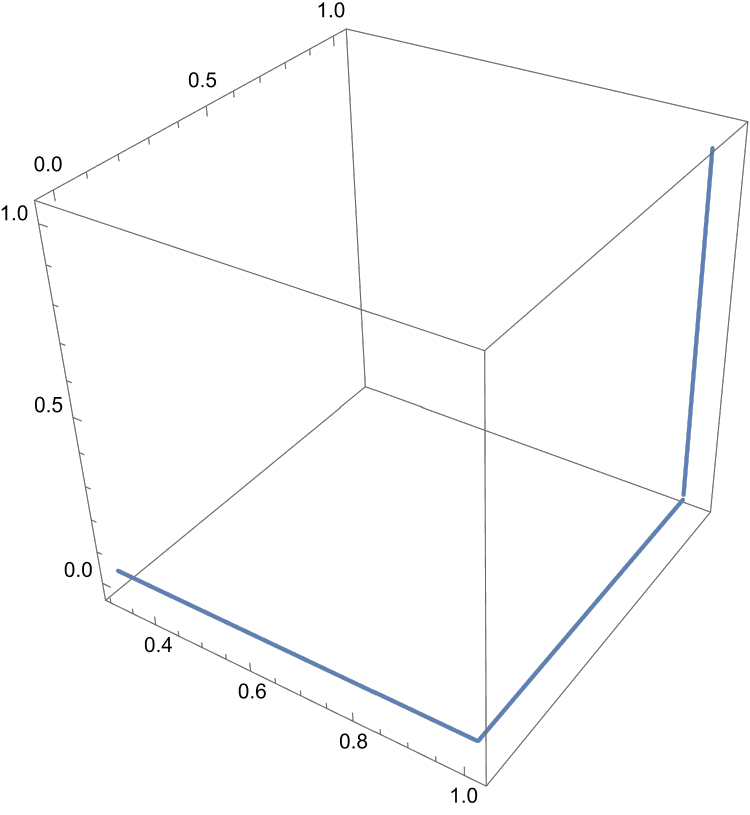}
\caption*{$m=(1,1,1)$}
\end{subfigure}
\caption{Images of $\mathsf{PwMom}^{m}$ for $m\in\comp(3)$ }
\label{fig:axis3_compositions}
\end{figure}

\begin{lemma}\label{lem:pwPolyDict}
Let $m\in \comp(M)$. The path $\mathsf{PwMom}^m$ forms a dictionary for $m$-splines of regularity $0$. 
\end{lemma}
\begin{proof}
We see that $m$-splines $X$ of regularity $0$ are precisely the concatenations of $\ell$ polynomial paths
$$X^{[1]} \sqcup \dots \sqcup X^{[\ell]}$$
such that $X^{[i]}$ has degree $m_i$. In particular, $X^{[i]} = A^{[i]}\circ \mathsf{Mom}^{m_i}$ . It is then easy to see that
\begin{align*}
    X = \begin{pmatrix} A^{[1]} & \dots & A^{[\ell]} \end{pmatrix} \circ \mathsf{PwMom}^m. \tag*{\qedhere}
\end{align*}
\end{proof}

\begin{theorem}\label{thm:coretensor}
For every $d,k\geq 2$ and $m\in\comp(M)$ the variety $\mathcal{S}_{d,\leq k,m}^0$ is described by the (sequence of) core tensors 
$$C:=\sigma^{\leq k}(\mathsf{PwMom}^m)\subseteq T^{\leq k}(\mathbb R^M),$$ 
that is, 
$\mathcal{S}_{d,\leq k,m}^0=\overline{\left\{A*C\mid A\in\mathbb{C}^{d\times M}\right\}}$. 
\end{theorem}
\begin{proof}
    With \Cref{lem:pwPolyDict} we know that $\mathsf{PwMom}^m$ is a dictionary for $m$-splines of regularity $0$, so every spline $X:[0,1]\rightarrow \mathbb R^d$ can be written as $X=A*\mathsf{PwMom}^m$ with $A\in\mathbb R^{d\times M}$. With equivariance \eqref{eq:equivariance_sig} we see that the image of $\sigma^{\leq k}$ is given by $\mathbb R^{d\times M}*C$, so the claim follows by taking Zariski closures over $\mathbb C$. 
\end{proof}

Take $m = (m_1,\ldots,m_l) \in \comp(M)$ and let us write $\alpha_1 \sqcup \ldots \sqcup \alpha_\ell = [M]$ for the associated partition. Let us give a closed form description of the core signature $\sigma^{\leq k}(\mathsf{PwMom}^m)$. We call a word $w$ \textit{adapted} to $m$, if we can write $w = w_1\ldots w_\ell$ such that $w_i$ is a word in the alphabet $[\alpha_i]$ for every $i$. Here we also allow the empty word. Note that this decomposition is unique. By subtracting $\sum_{j=1}^{i-1} m_j$ from all letters in $w_i$, $w$ determines words $v_i$ over the alphabet $[m_i]$ for every $1\leq i \leq \ell$. Let us write $v_i = v_{i1} \ldots v_{is_i}$, where the $v_{ij}$ are the letters of $v_i$.

\begin{proposition}\label{prop:closed-form-core}
    Let $w$ be a word. If $w$ is not adapted to $m$, then $\sigma(\mathsf{PwMom}^m)_{w} = 0$. Otherwise, we have $$\sigma(\mathsf{PwMom}^m)_{w} = \prod_{i=1}^\ell \prod_{j=1}^{s_i} \frac{v_{ij}}{\sum_{u=1}^j v_{iu}}$$
    where the $v_{ij}$ are defined as above.
\end{proposition}
\begin{proof}
    This follows from Chen's identity and \cite[Example 2.2]{AFS19}. 
\end{proof}

\begin{example}
For truncation $k=3$ and $m=(2,1)\in\comp(3)$ we illustrate the core tensor  $C=1\oplus C^{(1)}\oplus C^{(2)}\oplus C^{(3)}$ from \Cref{thm:coretensor} given by  
$$C^{(1)}=\begin{pmatrix}1\\1\\1\end{pmatrix},
\quad
C^{(2)}=\frac16\begin{pmatrix}
3&4&6\\
4&3&6\\
0&0&3
\end{pmatrix}
$$
and $3$-tensor 
$$
C^{(3)}=\frac1{60}\left(
\begin{array}{ccccccccccc}10 & 10 & 0 & \vrule & 15 & 16 & 0 & \vrule & 30 & 40 & 30 \\5 & 6 & 0 & \vrule & 8 & 10 & 0 & \vrule & 20 & 30 & 30 \\0 & 0 & 0 & \vrule & 0 & 0 & 0 & \vrule & 0 & 0 & 10\end{array}\right)
$$
in first matrix-tensor folding.

Let us verify the closed formula above for the entries of $C^{(2)}$. For $m=(2,1)$, the words not adapted to $m$ are precisely $\texttt{31}$ and $\texttt{32}$, explaining $C^{(2)}_{31} = C^{(2)}_{32} = 0$. For all other two-letter words $\texttt{ij}$, we have to find $v_1$ and $v_2$ from \Cref{prop:closed-form-core}. There are three possibilities: either $i,j\leq 2$, or $i\leq 2$ and $j = 3$, or $i,j = 3$. In the first case $v_2$ is the empty word, in the second case $v_2 = \texttt{1}$, and in the third case $v_2 = \texttt{11}$. From the first case, we obtain \begin{align*}
C^{(2)}_{11} = \frac{1}{1} \cdot \frac{1}{2}, \quad\quad
C^{(2)}_{12} =\frac{1}{1}\cdot \frac{2}{3}, \\C^{(2)}_{21} =\frac{2}{2} \cdot \frac{2}{3},\quad\quad
C^{(2)}_{22} =\frac{2}{2}\cdot \frac{2}{4}.
\end{align*}
From the second case we get $C^{(2)}_{13} = C^{(2)}_{23} =\frac{1}{1} \cdot \frac{1}{1}$. Finally, we obtain $C^{(2)}_{33} =\frac{1}{1} \cdot \frac{1}{2}$ from the third case.

\end{example}

\begin{proof}
    [Proof of \Cref{prop:SignatureIsMorphism}]
    For a fixed tensor $C$ and matrix $A$, the entries of the tensor $A * C$ are polynomials in the entries of $A$; so the statement follows by \Cref{thm:coretensor}.
\end{proof}

The natural next question is whether we can find core tensors for geometric and parametric $m$-splines of regularity $r>0$. Note that for fixed $\rho_{i,r}$, \eqref{eq:geom-cont} is a linear constraint on the columns of the parameter matrix. This will allow us to derive a dictionary for those geometric splines that satisfy \eqref{eq:geom-cont} with this specific choice of $\rho_{i,j}$, and in particular classes of parametric splines. To describe geometric splines, we introduce a generalization of dictionaries in \cite{pfeffer2019learning}. 
\begin{definition}
    Let $\mathcal X$ be some class of paths in $\mathbb R^d$. A \textit{parametrized dictionary} for $\mathcal X$ is a family of paths
    $$(\Psi_\rho: [0,1] \to \mathbb R^N)_{\rho \in \mathbb R^K}$$
    for some $N,K \in \mathbb N$ such that the paths in $\mathcal X$ are precisely the paths $A\Psi_\rho$ for $A \in \mathbb R^{d \times N}$ and $\rho$ in a Zariski dense subset of $\mathbb R^K$. We call such a dictionary \textit{algebraically parametrized} if the map $\rho \mapsto \sigma^{\leq k}(\Psi_\rho)$ is given by polynomials for all $k$.
\end{definition}

Instead of one core tensor, an algebraically parametrized dictionary defines a family of core tensors, and the Zariski closure of this family gives a \textit{core variety}. The signature variety is then the closure of the matrix orbit of the core variety.

Let $r, M\in \mathbb N$ with $M \geq r$ and let $m \in \comp_{\geq r}(M)$. We can now state the main result of this section. 
\begin{theorem}\label{thm:geom-spline-dict}
    The class of geometric $m$-splines of regularity $r$ admits an algebraically parametrized dictionary.
\end{theorem}

To construct this dictionary, we reinterpret \eqref{eq:geom-cont} as a matrix equation. In this way, we will obtain a matrix $B_\rho$ such that $B_\rho \circ \mathsf{PwMom}^m$ is the desired parametrized dictionary.

Write $m=(m_1,\ldots,m_l)$. Let $A$ be a $d \times M$ matrix. We introduce the following notation  for its columns:
$$\begin{pmatrix} A^{[1]}_1 & \dots & A^{[1]}_{m_1} & \dots & A^{[\ell]}_{1} & \dots & A^{[\ell]}_{m_\ell}\end{pmatrix}:=A$$

Now consider the path $A\circ \mathsf{PwMom}^m$. Plugging it into equation \eqref{eq:geom-cont}, we see that it is geometrically continuous if and only if we have
\begin{equation}\label{eq:geomcont-param}
    A^{[i+1]}_s = \rho_{i,j} \sum_{j=s}^{m_i} \binom{j}{s} A^{[i]}_j
\end{equation}
for all $i$ and $s \leq r$. 
In particular, the columns $A^{[i+1]}_s$ for $s\leq r$ are completely determined by the matrix $A^{[i]}$. For $i\geq 2$, let us define $\hat A^{[i]}$ as the $d\times (m_i - r)$ matrix obtained from $A^{[i]}$ by removing the first $r$ columns. We set
$$\hat A := \begin{pmatrix}
    \hat A^{[1]} & \dots & \hat A^{[\ell]}
\end{pmatrix}.$$
This is a $d \times \kappa$ matrix, where $\kappa := M - (\ell-1)\cdot r$.

Next, we recursively define $\kappa \times m_i$ matrices $B_\rho^{[i]}$ for $1\leq i \leq \ell$. We set
$$B_\rho^{[1]} := \begin{pmatrix}\mathrm{I}_{m_1} \\  0_{(\kappa - m_1) \times m_1} \end{pmatrix}$$
which serves as the base case of the recursion. 
For $1 \leq i \leq \ell-1$ and $s \leq r$ we then define columnwise
$$B_{\rho,s}^{(i+1)} := \rho_{i,s} B_\rho^{[i]} \begin{pmatrix}
    0 & \dots & 0 & \binom{s}{s} & \dots & \binom{m_i}{s}
\end{pmatrix}^\top$$
and for $r < s \leq m_i$
$$B_{\rho,s}^{[i+1]} := e_{s_0 + s}$$
where $s_0 := \sum_{j=1}^i m_j - i\cdot r$.
\begin{definition}
    We define the \textit{core spline transformation matrix}
    $$B_\rho := \begin{pmatrix}
    B_\rho^{[1]} & \dots & B_\rho^{[\ell]}
\end{pmatrix}.$$
\end{definition}
To indicate its dependence on $r$ and $m$ we will sometimes write $B^{m,r}_\rho$ for $B_\rho$ in the following.

From this we now obtain our desired dictionary.

\begin{proof}[Proof (of \Cref{thm:geom-spline-dict})]
    We claim that equation \eqref{eq:geomcont-param} is equivalent to the matrix equation $A= \hat A \cdot B_\rho$.
    
    Indeed, we can write $\hat A B_\rho$ as
    $$\begin{pmatrix}
        \hat A B_\rho^{[1]} & \dots & \hat A B_\rho^{[\ell]}
    \end{pmatrix}$$
    Thus $A = \hat A B_\rho$ if and only if $A^{[i]} = \hat A B_\rho^{[i]}$ for all $i$. Assume \eqref{eq:geomcont-param} holds; then we can use induction to show that $A^{[i]} = \hat A B_\rho^{[i]}$. For $i=1$ there is nothing to verify. Otherwise we have
    \begin{align*}(\hat A B_\rho^{[i]})_s &= \hat A B_{\rho,s}^{[i]} \\
    &= \rho_{i-1,s} \hat A B_{\rho,s}^{[i-1]} \begin{pmatrix}
    0 & \dots & 0 & \binom{s}{s} & \dots & \binom{m_{i-1}}{s}
\end{pmatrix}^\top \\
&= \rho_{i-1,s} A^{[i-1]} \begin{pmatrix}
    0 & \dots & 0 & \binom{s}{s} & \dots & \binom{m_{i-1}}{s} \\
\end{pmatrix}^\top \\
&= A^{[i]}_s
\end{align*}
for $s \leq r$, where we use the induction hypothesis in the third equality, and \eqref{eq:geomcont-param} in the last. For $s > r$ there is again nothing to show. Finally, reading the argument in reverse yields the converse implication.

It follows that $m$-splines of regularity $r$ are precisely given by $(\hat A B_\rho)\circ \mathsf{PwMom}^m$ for a $d\times \kappa$-matrix $\hat A$ and some $\rho>0$, where $\kappa = M - (s-1)r$. Thus, by associativity, $(B_\rho \circ \mathsf{PwMom}^m)_\rho$ is a parametrized dictionary for this class of paths. Since the entries of $B_\rho$ are polynomials in $\rho$, equivariance of the signature implies that $(B_\rho \circ \mathsf{PwMom}^m)_{\rho}$ is indeed an algebraically parametrized dictionary.
\end{proof}
\begin{corollary}\label{cor:dictForParametricSplines}
    $B^{m,r}_1 \circ \mathsf{PwMom}^m$ is a dictionary for parametric $m$-splines of regularity $1$.
\end{corollary}
\begin{corollary}\label{cor:coretensorforParametricSplines}
Let $C$ be the (sequence of) core tensors from \Cref{thm:coretensor} for piecewise polynomial paths. 
\begin{enumerate}[(i)]
    \item The variety $\mathcal{S}_{d,k,m}^r$ is described by the parametrized core tensor
    $B^{m,r}_\rho*C$. 
    
    \item The variety $\mathcal{P}_{d,k,m}^r$ is described by the core tensor 
    $B^{m,r}_1*C$. 
\end{enumerate}
\end{corollary}
\begin{proof}
    This follows immediately from \Cref{thm:geom-spline-dict} and equivariance \eqref{eq:equivariance_sig} of the signature. 
\end{proof}
In \Cref{fig:dictionariesForParametricSplines} we illustrate two dictionaries for parametric $m$-splines with regularity $r=1$, according to \Cref{cor:dictForParametricSplines}.

\begin{figure}[h]
    \centering
    \begin{subfigure}{0.3\linewidth}
        \centering
        \includegraphics[width=0.7\linewidth]{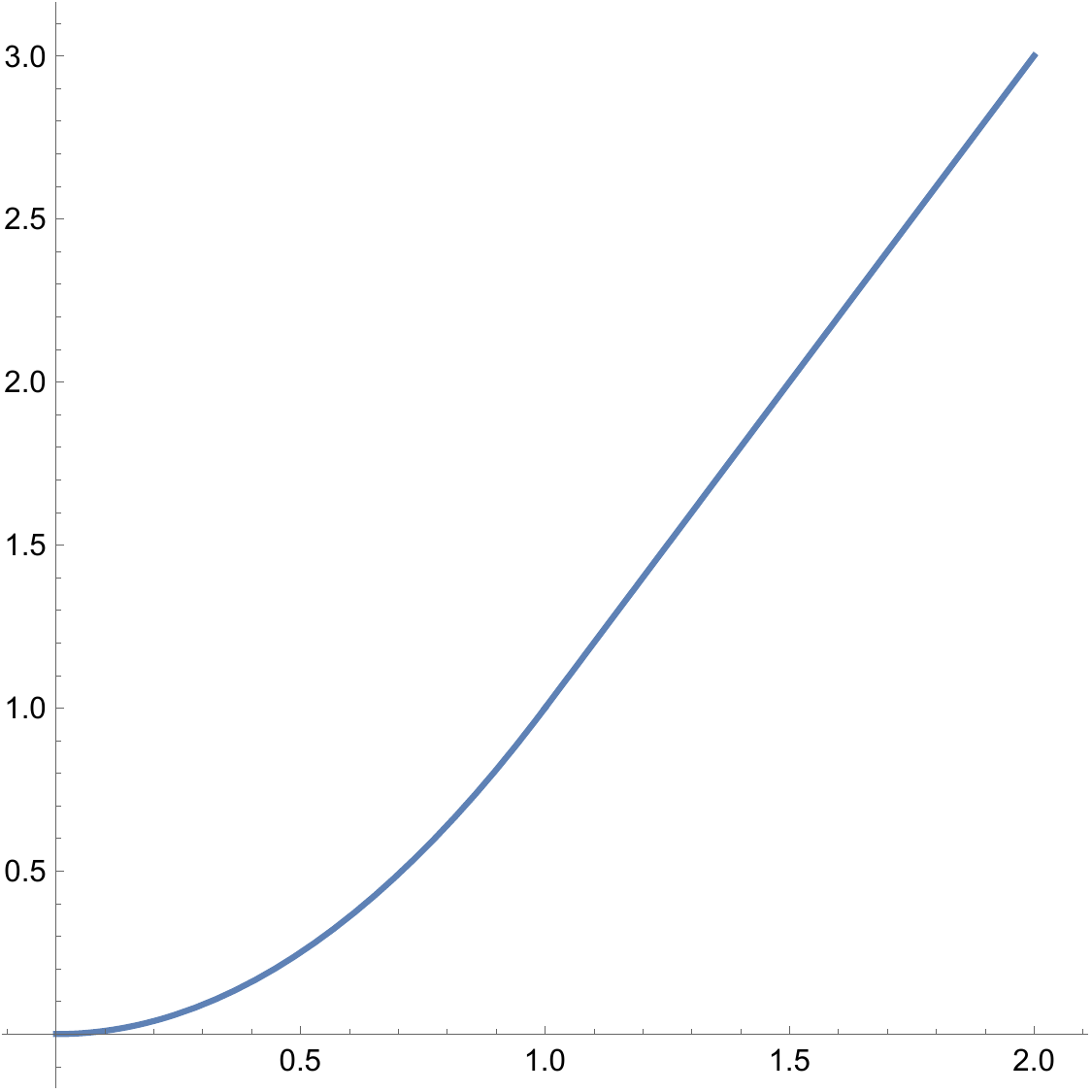}  
        \caption{$m=(2,1)$}
    \end{subfigure}
     \begin{subfigure}{0.3\linewidth}
        \centering
        \includegraphics[width=0.7\linewidth]{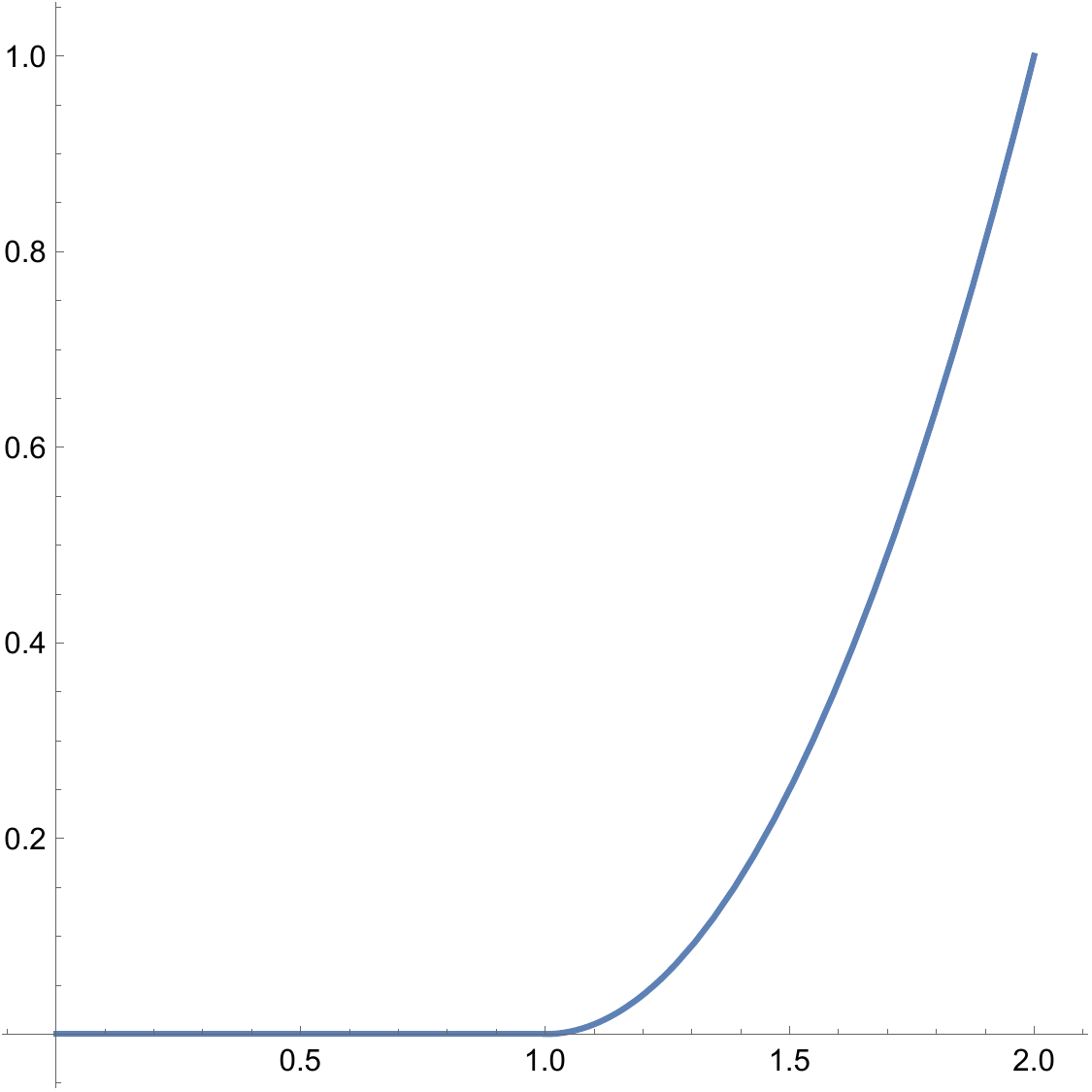}  
        \caption{$m=(1,2)$}
    \end{subfigure}
    \caption{Dictionaries for $\mathcal{P}^{1}_{2,\leq k,m}$}
    \label{fig:dictionariesForParametricSplines}
\end{figure}

\begin{example}
    Let us determine the core variety for truncation level $2$ and geometric $(2,1)$-splines of regularity $1$. The associated parametrized dictionary is $B_\rho^{(2,1)}\circ \mathsf{PwMom}^m$ which is the class of paths
    $$X_\rho := \begin{pmatrix} t \\ t^2 \end{pmatrix} \sqcup \rho \begin{pmatrix} t \\ 2t \end{pmatrix}.$$
    We have
    \begin{align*}
        \sigma^{(1)}(X_\rho) &= \begin{pmatrix} 1 + \rho \\ 1 + 2\rho \end{pmatrix} \\
        \sigma^{(2)}(X_\rho) &= \begin{pmatrix} \frac{1}{2}\rho^2 + \rho + \frac{1}{2} & \rho^2 + 2\rho + \frac{2}{3} \\ \rho^2 + \rho + \frac{1}{3} & 2\rho^2 + 2\rho + \frac{1}{2} \end{pmatrix}
    \end{align*}
    This parametrizes a curve in the $3$-dimensional space $\mathcal G^{\leq 2}(\mathbb R^2)$. Using Lyndon coordinates $x_1,x_2,x_{12}$, it is the Zariski image of the map
    $$\rho \mapsto (1 + \rho, 1 + 2\rho, \rho^2 + 2\rho + \frac{2}{3}).$$
    Using \textsc{Macaulay 2}, we see that the image is the intersection of the quadric $3x_2^2 + 6x_2 - 12x_{12} = 1$ with the affine hyperplane $2x_1 - x_2 = 1$. The orbit closure of this variety under the action of $2\times 2$-matrices on $\mathcal G^{\leq 2}(\mathbb R^2)$ is the whole space $\mathcal G^{\leq 2}(\mathbb R^2)$.
\end{example}

\section{Dimension and degree of spline varieties}

Using the dictionaries constructed in the last section, we now study dimension, degree and generators of the spline varieties we defined.

\subsection{Matrix varieties}

In \cite[Theorem 3.4]{AFS19}, it was shown that the varieties of polynomial paths with degree $M$ and piecewise linear paths with $M$ segments agree, that is $$\mathcal{M}_{d,M} := \mathcal S^0_{d,2,(M)} = \mathcal S^0_{d,2,m}$$
with $m = (1,\ldots,1)\in\comp(M)$. We generalize this result to our setting.

\begin{theorem}\label{thm:sig-matrices}
    For $m \in \comp_{\geq r}(M)$, the matrix varieties $\mathcal S^{r}_{d,2,m}$ and $\mathcal P^{r}_{d,2,m}$ are identical and agree with the variety $\mathcal M_{d,M-(\ell - 1)r}$. 
\end{theorem}

The proof relies on the following linear algebra lemma.
\begin{lemma}\label{lem:sig-matrices}
    Let $A\in\mathbb C^{d\times d}$ be a rank $n$ matrix such that $A^{\mathsf{sym}}$ has rank $1$. Let
    $$\tau := \begin{pmatrix} 0 & 1 \\ -1 & 0 \end{pmatrix} \quad \text{and} \quad \omega := \begin{pmatrix} 1 & 1 \\ -1 & 0 \end{pmatrix}$$
    Then $A$ is congruent to $(1) \oplus \tau^{\frac {n-1} 2} \oplus (0)^{\oplus d - 1 - \frac {n-1} 2}$ if $n$ is odd,  and it is congruent to $\omega \oplus \tau^{\frac {n-2} 2} \oplus (0)^{\oplus d - 1 - \frac {n} 2}$ if $n$ is even.
\end{lemma}
\begin{proof}
    By \cite[Theorem 1]{HORN20061010}, $A$ is congruent to a direct sum of blocks $J_n(0)$, $\Gamma_n$ and $H_{2n}(\mu)$. We refer to loc.\ cit.\ for the definitions. Among these blocks, only $H_2(-1) = \tau$ is skew-symmetric. Moreover, only for the two blocks $\Gamma_1 = (1)$ and $\Gamma_2 = \omega$ the symmetric part has rank $1$. All three blocks have full rank. Altogether, this implies the statement.
\end{proof}
    We can now prove the theorem.
\begin{proof}[Proof (of \Cref{thm:sig-matrices})]
    Let $\rho \in \mathbb R^{\ell - 1, r}$ and set
    $$C := \sigma^{(2)}(B_\rho \mathsf{PwMom}^m) = B_\rho \sigma^{(2)}( \mathsf{PwMom}^m)B_\rho^\top.$$
    As $C$ is a signature matrix, its symmetric part has rank $1$, see \cite[Section 21]{AFS19}. Thus, \Cref{lem:sig-matrices} applies. The matrix $\sigma^{(2)}(\mathsf{PwMom}^m)$ has full rank: indeed, by \Cref{prop:closed-form-core}, the matrix is upper block triangular, with every block on the diagonal of the form $\sigma^{(2)}(\mathsf{Mom}^{m_i})$. These blocks have full rank by \cite[Equation (21)]{AFS19}, such that the rank of $C$ is simply the rank of $B_\rho$. But the $(M - (\ell -1)r) \times M$-matrix $B_\rho$ has full rank as well; indeed, by construction every standard basis vector $e_i \in \mathbb R^{M - (\ell -1)r}$ appears as a column of $B_\rho$.\par
    Set $n := M - (\ell -1)r$. By \Cref{cor:coretensorforParametricSplines} it follows that both $\mathcal S^{r}_{d,2,m}$ and $\mathcal P^{r}_{d,2,m}$ are given by congruence orbits of $(1) \oplus \tau^{\frac {n-1} 2} \oplus (0)^{\oplus d - 1 - \frac {n-1} 2}$ if $n$ is odd, or $\omega \oplus \tau^{\frac {n-2} 2} \oplus (0)^{\oplus d - 1 - \frac {n} 2}$ if $n$ is even. But these are precisely the varieties $\mathcal M_{d,M-(\ell -1)r}$.
\end{proof}

\begin{corollary}For all $m\in\comp_{\geq r}(M)$, 
    $$\dim(\mathcal{M}_{d,m}) = Md-(\ell-1)dr - \binom{M-(\ell-1)r}{2}.$$
\end{corollary}
\begin{proof}
    Combine \cite[Theorem 3.4]{AFS19} and \Cref{thm:sig-matrices}.
\end{proof}

Using the determinantal description of $\mathcal M_{d,M}$ from \cite[Theorem 3,4]{AFS19}, we can determine the degree of the variety in the case where $M$ is odd.

The variety $\mathcal{M}_{d,1}$ is a quadratic Veronese, so that we obtain $\deg(\mathcal{M}_{d,1}) = 2^{d-1}$. More generally, we have the following. 

\begin{theorem}\label{cor:HARRIS198471}
    If $M$ is odd, then
   $$\deg(\mathcal{M}_{d,M}) = 2^{M-1} \prod_{n = 0}^{d-M-1} \frac{\binom{d+n}{d-M-n}}{\binom{2n+1}{n}}$$
\end{theorem}
\begin{proof}
    The intersection of the space of matrices with symmetric part rank $1$ with the space of matrices with skew-symmetric part rank $\leq M$ is transverse, as symmetric and skew-symmetric parts define independent coordinates on the matrix space. Thus, the degree of $\mathcal M_{d,M}$ is just the product of the degree of the variety of symmetric matrices of rank $1$ times the degree of the variety of skew-symmetric matrices of rank $\leq M$. The former degree is $\deg(\mathcal{M}_{d,1}) = 2^{d-1}$ and the latter is computed in \cite[Proposition 12]{HARRIS198471}.
\end{proof}

As a special case, we obtain the degree of $\mathcal{M}_{d,3}$, whose formula had been observed in \cite[page 15]{AFS19}.

\begin{corollary}
   For all $d\geq 3$, 
$$
    \deg(\mathcal{M}_{d,3}) = \frac{2^{d-1}}{d-1}\binom{2d-4}{d-2}
$$
\end{corollary}
\begin{proof}
    By \Cref{cor:HARRIS198471}, the claim reduces to $$\frac{1}{d-1}\binom{2d-4}{d-2} = \frac{1}{2^{d-3}} \prod_{n = 0}^{d-4} \frac{\binom{d+n}{d-3-n}}{\binom{2n+1}{n}}$$ which can be shown by induction on $d$. For $d=3$, both expressions evaluate to $1$. Going from $d-1$ to $d$, the left hand side is multiplied by $\frac{d-2}{d-1} \cdot \frac{(2d-5)(2d-4)}{(d-2)^2} = 2  \frac{2d-5}{d-1}$, and the right hand side by
    $$\frac{\prod_{n=0}^{d-4} \frac{d+n}{d-1-n} }{2\binom{2(d-4)+1}{d-4}} = \frac{\binom{2d-4}{d-1}}{2\binom{2d-7}{d-4}}.$$
    Note that
    $$\binom{2d-4}{d-1} = \frac{(2d-4)(2d-5)(2d-6)}{(d-1)(d-2)(d-3)} \cdot \binom{2d-7}{d-4} = 4 \frac{2d-5}{d-1} \cdot \binom{2d-7}{d-4}.$$
    Thus, the right hand side is multiplied with $2  \frac{2d-5}{d-1}$ as well and we conclude.
\end{proof}

\subsection{Higher tensor varieties}

To compute signature varieties for $k>2$ we use computational methods.

\begin{example}
Let us compute the spline varieties $\mathcal S^0_{2,\leq 4,m}$ for $m \in \comp(3)$. The two compositions $(1,1,1)$ and $(3)$ correspond to piecewise linear and polynomial paths, whose projected signature varieties $\mathcal S^0_{d,m,k}$ were already considered in \cite{AFS19}. We compute the ideals of the varieties $\mathcal S^0_{d,m,\leq k}$ in the coordinate ring of $\mathcal G^{\leq k}(\mathbb C^d)$, see also \Cref{rem:lyndon-coordinates}.
 \par
All three varieties have dimension $6$, which is the number of parameters. The variety $\mathcal S_{d,(1,1,1),\leq 4}^0$ has codimension $2$ in the $8$-dimensional ambient space $\mathcal G^{\leq k}(\mathbb C^d)$. It is the complete intersection of one octic and one decic. The octic is given by the following polynomial, 
\begin{align*}
x_{1}^2x_{2}^2x_{1 2}^2 &+6x_{1 2}^4-12x_{2}x_{1 2}^2x_{1 1 2}+18x_{2}^2x_{1 1 2}^2-12x_{1}x_{12}^2x_{1 2 2}\\
&-36x_{1}x_{2}x_{1 1 2}x_{1 2 2}+18x_{1}^2x_{1 2 2}^2-24x_{2}^2x_{1 2}x_{1 1 12}\\
&+144x_{2}x_{1 2 2}x_{1 1 1 2}-6x_{1}^2x_{2}^2x_{1 1 2 2}+48x_{1}x_{2}x_{1 2}x_{1 1 22}\\
&-72x_{2}x_{1 1 2}x_{1 1 2 2}-72x_{1}x_{1 2 2}x_{1 1 2 2}+72x_{1 1 2 2}^2\\
&-24x_{1}^2x_{12}x_{1 2 2 2}+144x_{1}x_{1 1 2}x_{1 2 2 2}-288x_{1 1 1 2}x_{1 2 2 2}. 
\end{align*}\par
Note that the variables are graded according to \Cref{rem:lyndon-coordinates}.

The variety $\mathcal S_{d,(3),\leq 4}^0$ is cut out by three polynomials, in degrees $16,18$ and $20$.\par
Finally, we see that $\mathcal S_{d,\leq 4,(2,1)}^0$ is cut out by four polynomials, one of degree $16$, another of degree $20$, and the remaining two of degree $18$.\par
We can also compute the degrees of the varieties $\mathcal S_{d,4,(2,1)}$ that we obtain under the projection $\mathcal G^{\leq 4}(\mathbb R^2) \to (\mathbb R^2)^{\otimes 4}$. We see that for $m=(1,1,1)$ and $m=(3)$ we obtain the degrees $64$ and $192 = 3 \cdot 64$ that were already observed in \cite[Table 3]{AFS19}. Interestingly, the degree $128 = 2\cdot 64$ of the variety $\mathcal S_{2,4,(2,1)}$ interpolates between the two.
\end{example}

 \Cref{tab:tensor-var-d2k4} and \Cref{tab:tensor-var-d3k3} summarize dimension and degree for different signature varieties of geometric and parametric splines. Here, we view the signature varieties as embedded into truncated tensor space $T^{\leq k}(\mathbb C^d)$.

\begin{table}[ht!]
    \centering 
    \begin{subtable}{0.49\linewidth}
    \centering
        \begin{tabular}{c|ccc}
        $m$ $\backslash$ $r$ & 0 & 1 & 2 \\ \hline
         $(2,2)$ & 8 & $7^{\mathcal{S}},6^{\mathcal{P}}$ & $6^{\mathcal{S}},4^{\mathcal{P}}$ \\
         $(2,1,1)$ & 8 & $5^{\mathcal{S}},4^{\mathcal{P}}$ & - \\
        $(2,1)$ & 6 & $5^{\mathcal{S}},4^{\mathcal{P}}$ & - \\
         $(1,1,1)$ & 6 & $2^{\mathcal{S}},2^{\mathcal{P}}$ & -
        \end{tabular}
        \caption{Dimension}
    \end{subtable}%
    \begin{subtable}{0.49\linewidth}
    \centering
        \begin{tabular}{c|ccc}
        $m$ $\backslash$ $r$ & 0 & 1 & 2 \\ \hline
         $(2,2)$ & 1 & ? & $? ,96^{\mathcal{P}} $\\
         $(2,1,1)$ & 1 & $414^{\mathcal{S}}, 96^{\mathcal{P}}$ & - \\
        $(2,1)$ & ? & $414^{\mathcal{S}}, 96^{\mathcal{P}}$  & - \\
         $(1,1,1)$ & $276$&  $16^{\mathcal{S}}, 16^{\mathcal{S}}$ & -
        \end{tabular}
        \caption{Degree (in $T^{\leq 4} (\mathbb C^d)$)}\label{tab:tensor-var-d2k4-deg}
    \end{subtable}
    \caption{Dimension and degree for $\mathcal{S}^{r}_{2,\leq 4,m}$ and $\mathcal{P}^{r}_{2,\leq 4,m}$}
    \label{tab:tensor-var-d2k4}
\end{table}

 Note that the ambient space $\mathcal G^{\leq 4}(\mathbb C^2)$ of the varieties considered in \Cref{tab:tensor-var-d2k4} has dimension $8$. With a growing number of variables, symbolic computations became infeasible, indicated by question marks in \Cref{tab:tensor-var-d2k4-deg}. In this case, we used the \textsc{Macaulay 2} package \texttt{NumericalImplicitization} \cite{NumericalImplicitizationArticle} to compute the missing dimensions. We were not able to determine the missing degrees numerically.

\begin{table}[ht!]
    \centering
    \begin{tabular}{c|ccc}
        $m$ $\backslash$ $r$ & 0 & 1 & 2 \\ \hline
         $(3,2)$ & 14 & $13^{\mathcal{S}}, 12^{\mathcal{P}}$ & $11^{\mathcal{S}}, 9^{\mathcal{P}}$ \\
         $(3,1,1)$ & 14 & $10^{\mathcal{S}},9^{\mathcal{P}}$ & - \\
        $(2,2)$ & 12 & $10^{\mathcal{S}},9^{\mathcal{P}}$ & $7^{\mathcal{S}},6^{\mathcal{P}}$ \\
         $(2,1,1)$ & 12 & $7^{\mathcal{S}},6^{\mathcal{P}}$ & - \\
          $(2,1)$ & 9 & $7^{\mathcal{S}},6^{\mathcal{P}}$& - \\
    \end{tabular}
    \caption{Dimension for $\mathcal{S}^{r}_{3,\leq 3,m}$ and $\mathcal{P}^{r}_{3,\leq 3,m}$}
    \label{tab:tensor-var-d3k3}
\end{table}

In \Cref{tab:tensor-var-d3k3}, the ambient space $\mathcal G^{\leq 3}(\mathbb C^3)$ is 14-dimensional. Again, we computed the dimensions using \texttt{NumericalImplicitization}. 

\begin{remark}
    For most of the varieties in \Cref{tab:tensor-var-d2k4} and \Cref{tab:tensor-var-d3k3}, the dimension is either the dimension of the ambient space or the number of parameters. This is in line with the observations from \cite{AFS19}; cf.\ Conjecture 6.10 in loc.\ cit. However, there is one notable exception: the variety $\mathcal S^{2}_{3,3,(2,2)}$ has dimension $7$, even though the number of parameters is $2\cdot 3 + 2 = 8$.
\end{remark}

\section{Spline recovery}

An important inverse problem in rough analysis is the task of recovering a path from Chen’s iterated-integrals signature; see \cite{lyons2017hyperbolic} for a general overview. This task is also known as the \emph{learning problem}. The influential work \cite{pfeffer2019learning} formalized this problem via group orbits under matrix-tensor congruence, and it establishes identifiability results, both exact and numerical, for piecewise linear,  polynomial, and generic paths.
For truncation level $k=3$ and piecewise linear paths with $M=d$ segments, there is an efficient algorithm for path recovery \cite[Theorem 1.1 and Algorithm 3.1]{schmitz2025efficient}. In this case the learning problem is known to have a unique solution; see \cite[Theorem 6.2]{pfeffer2019learning}. 
\par
We study the learning problem for splines, i.e., we investigate the fibers of the signature map restricted to geometric or parametric $m$-splines for a fixed truncation $k$. Contrary to the setting above, the learning problem is not (uniquely) solvable in general. This motivates the following notion of path recovery degree. 

\begin{definition}\label{def:recDeg}
   For $d,k,r\in\mathbb{N}$ and $m\in\comp_{\geq r}(M)$, the \emph{recovery degree for geometric $m$-splines} 
   $$\mathsf{PRdeg}(\mathcal{S}_{d,\leq k,m}^r)\in\mathbb N \cup \{ \infty \}$$ is given by the number of (complex) points in a generic fiber of $\sigma^{\leq k}$ on $\mathcal{S}_{d,\leq k,m}^r$. Analogously, we define the \emph{recovery degree for parametric $m$-splines}, denoted by $\mathsf{PRdeg}(\mathcal{P}_{d,\leq k,m}^r)$.
\end{definition}

\begin{example}
    Consider the class of geometric $(2,1)$-splines of regularity $1$ in $\mathbb R^2$. According to \Cref{thm:geom-spline-dict} it is parametrized by a $2\times 2$-matrix $A$ and a scalar $\rho$. The associated map
    \begin{align*}
        \mathbb R^{2\times 2} \times \mathbb R &\to \mathcal G^{\leq k}(\mathbb R^2) \\
        (A, \rho) &\mapsto A B_\rho * \sigma^{\leq 3}(\mathrm{PwMom}^{(2,1)})
    \end{align*}
    parametrizing $\mathcal S^1_{2,\leq 3,(2,1)}$ is generically finite-to-one. A generic fiber of this map contains two points, i.e., $\mathsf{PRdeg}(\mathcal S^1_{2,\leq 3,(2,1)}) = 2$.

    We are particularly interested in the real points of the fiber, as those correspond to path solutions. Among these, only the points with $\rho > 0$ correspond to spline reconstructions.

    By computing a Gröbner basis of the generic fiber ideal with respect to an elimination order, we see that the two points of the fiber correspond to the two solutions of a quadratic polynomial $q$ in $\rho$, given by $\rho^{2}+c\rho + \frac{1}{6}c$ with
    $$c = \frac{-x_{1}^{2}x_{2}^{2}+24\,x_{1}x_{12}x_{2}+36\,x_{12}^{2}-60\,x_{1}x_{122}-60\,x_{112}x_{2}}{30\,x_{1}x_{12}x_{2}+60\,x_{12}^{2}-90\,x_{1}x_{122}-90\,x_{112}x_{2}}$$
    where we denoted the coordinates on $\mathcal G^{\leq k}(\mathbb R^2)$ by $x_i$ (see \Cref{rem:lyndon-coordinates}). We remark that as an element of the tensor algebra, the denominator simplifies to
    $$30 (2 \cdot e_1 \otimes e_2 \otimes e_1 \otimes e_2 - 5 \cdot e_1 \otimes e_1 \otimes e_2 \otimes e_2).$$
    One checks that the quadratic polynomial $q$ has at most one positive real solution. That is, the truncated signature $\sigma^{\leq 3}$ characterizes geometric $(2,1)$-splines of regularity $1$ in the plane uniquely. \Cref{fig:m21-fiber-paths} shows the two points for a specific fiber of the signature map: one of them is a spline, the other one is a `cusp' solution where $\rho < 0$.
\end{example}

\begin{figure}[ht!]
    \centering
    \includegraphics[width=0.4\linewidth]{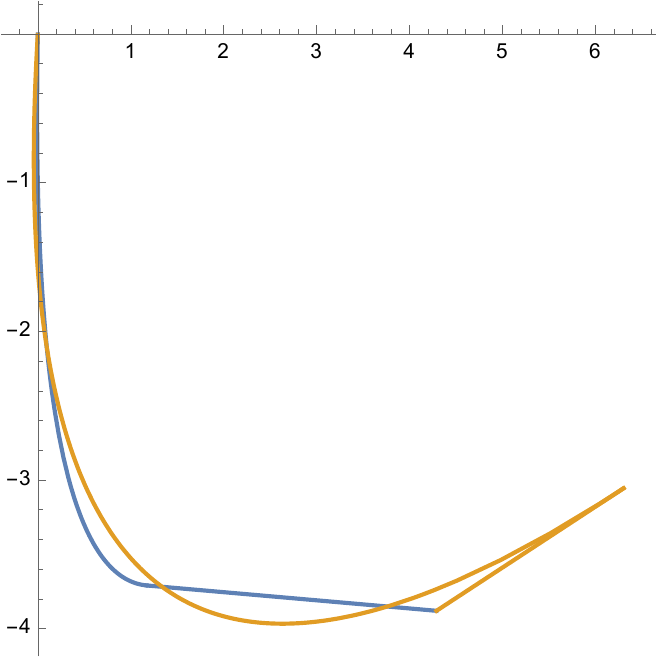}
    \caption{Two paths witnessing $\mathsf{PRdeg}(\mathcal S^r_{2,\leq 3,(2,1)}) = 2$.}\label{fig:m21-fiber-paths}
\end{figure}

Similar to \cite{pfeffer2019learning}, we study the recovery degree in general by determining the ideal of a generic fiber. The degrees of these ideals can then be related to path recovery degrees; see \Cref{prop:recoverydegviaIdeal}. For this construction we use the core tensors according to \Cref{cor:coretensorforParametricSplines}.

\begin{definition}\label{prop:pathrecoveryIdeal_params}
Let $C_\rho$ be the family of core tensors for $\mathcal{S}_{d,k,m}^{r}$.  For every $A\in\mathbb{R}^{d\times (M-rd)},\eta \in \mathbb R^{r \times (\ell-1)}$ we have 
   \begin{align*} I_{A,\eta}
   &
   :=\left\langle (A*C_\eta-a*C_\rho)_w\mid w\right\rangle\label{eq:param_normal}
   \end{align*}
   where $a=(a_{i,j})_{1\leq i\leq d,1\leq j\leq (M-rd)}$ and $\rho=(\rho_{i,j})_{1\leq i\leq r,1\leq j\leq (\ell-1)}$ are matrices of our $d(M-rd)+r(\ell-1)$ polynomial ring variables. 
\end{definition}

Let $I_{A,\eta}$ be the ideal from \Cref{prop:pathrecoveryIdeal_params} for generic $A$ and $\eta$. If $\dim I_{A,\eta} = 0$, then
\begin{equation}\label{prop:recoverydegviaIdeal}
    \mathsf{PRdeg}(\mathcal{S}_{d,\leq k,m}^r)=\deg (I_{A,\eta}).
\end{equation}
Similarly, if $\dim(J_{A}) = 0$, we set $J_{A} := I_{A,1} + \langle\rho_{ij} - 1\mid i,j\rangle$ and have 
$$\mathsf{PRdeg}(\mathcal{P}_{d,\leq k,m}^r)=\deg (J_{A}).$$

The simplest setting is again for regularity $r=0$. Then geometric and parametric $m$-splines are piecewise polynomial paths with segments of degrees $m_i$. In \Cref{tab:possibleRecoveries_d2k4_8param_r0} we present all recovery degrees for $m\in\comp(4)$ with $d=2$ and $k=4$. This refines \cite[Remark 6.13]{AFS19}, where the recovery degrees were presented for piecewise linear and polynomial paths, corresponding in our language to $(1,1,1,1),(4)\in\comp(4)$, respectively. 
We use the construction of $J_{A}$ according to \eqref{prop:recoverydegviaIdeal} and compute its degree. 

\begin{table}[ht!]
    \centering
    $\begin{array}{l|l}
      m\in\comp(4) & \mathsf{PRdeg}(\mathcal{S}_{2,\leq 4,m}^0)\\\hline
   (1,1,1,1)& 4  \\\hline
   (1,2,1)  & 18 \\\hline
   (2,1,1), (1,1,2) & 14 \\\hline
   (2,2)    & 10  \\\hline
    (3,1),(1,3)  & 40 \\\hline
   (4)     & 48
\end{array}$
\caption{All recovery degrees for $\mathcal{S}_{2,\leq 4,m}^0$ and $M=4$.
\label{tab:possibleRecoveries_d2k4_8param_r0}}
\end{table}

We move on to parametric spline recovery. The vanishing ideals from \Cref{prop:pathrecoveryIdeal_params} are zero-dimensional if and only if the number of parameters agrees with $\dim(\mathcal P^r_{d,\leq k,m})$. In this case, we observed in all examples we computed that $\mathsf{PRdeg}(\mathcal P^r_{d,\leq k,m}) = 1$ if $\mathcal P^r_{d,\leq k,m}$ is a strict subvariety of $\mathcal G^{\leq k}(\mathbb R^d)$. The path recovery degree becomes larger in the case where $\mathcal P^r_{d,\leq k,m} = \mathcal G^{\leq k}(\mathbb R^d)$ . \Cref{tab:possibleRecoveries_d2k4_8param} enumerates all choices of $m$ and $r$ such that the number of parameters $M-r(\ell -1)$ equals $\dim(\mathcal{P}_{2,\leq 4, m}) = \dim(\mathcal{G}^{\leq 4}(\mathbb R^2))=8$. We computed the path recovery degree in each case, grouping together all $r$ and $m\in\comp_{\geq r}(M)$ with coinciding degrees. According to  \Cref{prop:antipode} we can omit all redundant compositions based on reverse compositions and insertions of $r$ in $m$ with regularity $r$. We use the construction of $I_{A,\rho}$.

\begin{table}[ht!]
    \centering
    $\begin{array}{l|l|l}
      m\in\comp(M) & r &  \mathsf{PRdeg}(\mathcal{P}_{2,\leq 4,m}^r)\\\hline
   (2,2,2) ,  (2,1,2,1,2)    &1& 46  \\
     (1,2,2,2,1),  (1,2,1,2,1,2,1)   &&  \\\hline
      (2,2,2,1),(2,1,2,1,2,1)    &1& 50  \\\hline
     (2,2,1,2) ,(2,2,1,2,1)   &1& 52  \\
    (2,1,2,2,1),(1,2,2,1,2,1)   &&   \\\hline
      (1,4,1)   &1& 54   \\\hline
       (4,1)  ,(3,2)   , (3,2,1)  ,  (3,1,2) ,     &1& 56  \\
    (2,3,1), (3,1,2,1), (2,1,3,1)       &&  \\
    (1,3,2,1), (1,3,1,2,1)     &&  \\\hline
      (3,3) ,(3,2,3), (2,4,2)    &2 & 54  \\
     (2,3,3,2),(2,3,2,3,2) && \\\hline
         (4,2) ,(3,3,2) ,(3,2,3,2)  & 2& 60 \\\hline
           (3,4,3)     &3 & 54  \\\hline
      (4,3)    &3 & 60 
\end{array}$
\caption{Recovery degrees for parametric $m$-splines in the plane from the $4$-truncated signature
\label{tab:possibleRecoveries_d2k4_8param}}
\end{table}

Similarly, we compute recovery degrees for geometric spline recovery. Again, in the cases where $\mathcal S^r_{d,\leq k,m}$ is a strict subvariety of $\mathcal G^{\leq k}(\mathbb R^d)$, we observed $\mathsf{PRdeg}(\mathcal S^r_{d,\leq k,m}) = 1$. In the cases where $\mathcal S^r_{d,\leq k,m}$ attains the maximal dimension but is not overparametrized, \Cref{tab:rec-geom-plane} and \Cref{tab:rec-geom-space} show recovery degrees for $m$-splines in the plane and in $3$-space, for different choices of $m$ and $r$ and truncation levels $k$. Here, our recovery computations relied heavily on the F4 implementation provided by \texttt{msolve} \cite{msolve21}. 

\begin{table}[ht!]
    \centering
    $\begin{array}{l|l|l}
      m\in\comp(M) & r & \mathsf{PRdeg}(\mathcal{S}_{2,\leq 4,m}^r)\\\hline
   (2,2,1) & 1 & 32\\
   (2,1,2) &  &  \\\hline
   (1,3,1) & 1 & 96 \\\hline
    (3,2) & 2 & 116 
\end{array}$
\caption{Recovery degrees for geometric $m$-splines in the plane from the $4$-truncated signature\label{tab:rec-geom-plane}}
\end{table}

\begin{table}[ht!]
    \centering
    $\begin{array}{l|l|l}
      m\in\comp(M) & r & \mathsf{PRdeg}(\mathcal{S}_{3,\leq 3,m}^r)\\\hline
   (3,2,1) & 1 & 144\\\hline
   (3,1,2) & 1 & 84\\\hline
   (2,2,2) & 1 & 90\\\hline
   (4,2) & 2 & 312\\\hline
   (3,3) & 2 & 168
\end{array}$
\caption{Recovery degrees for geometric $m$-splines in $3$-space from the $3$-truncated signature\label{tab:rec-geom-space}}
\end{table}

\begin{figure}[ht!]
    \centering
    \begin{subfigure}{0.27\columnwidth}
        \centering
\includegraphics[width=\linewidth]{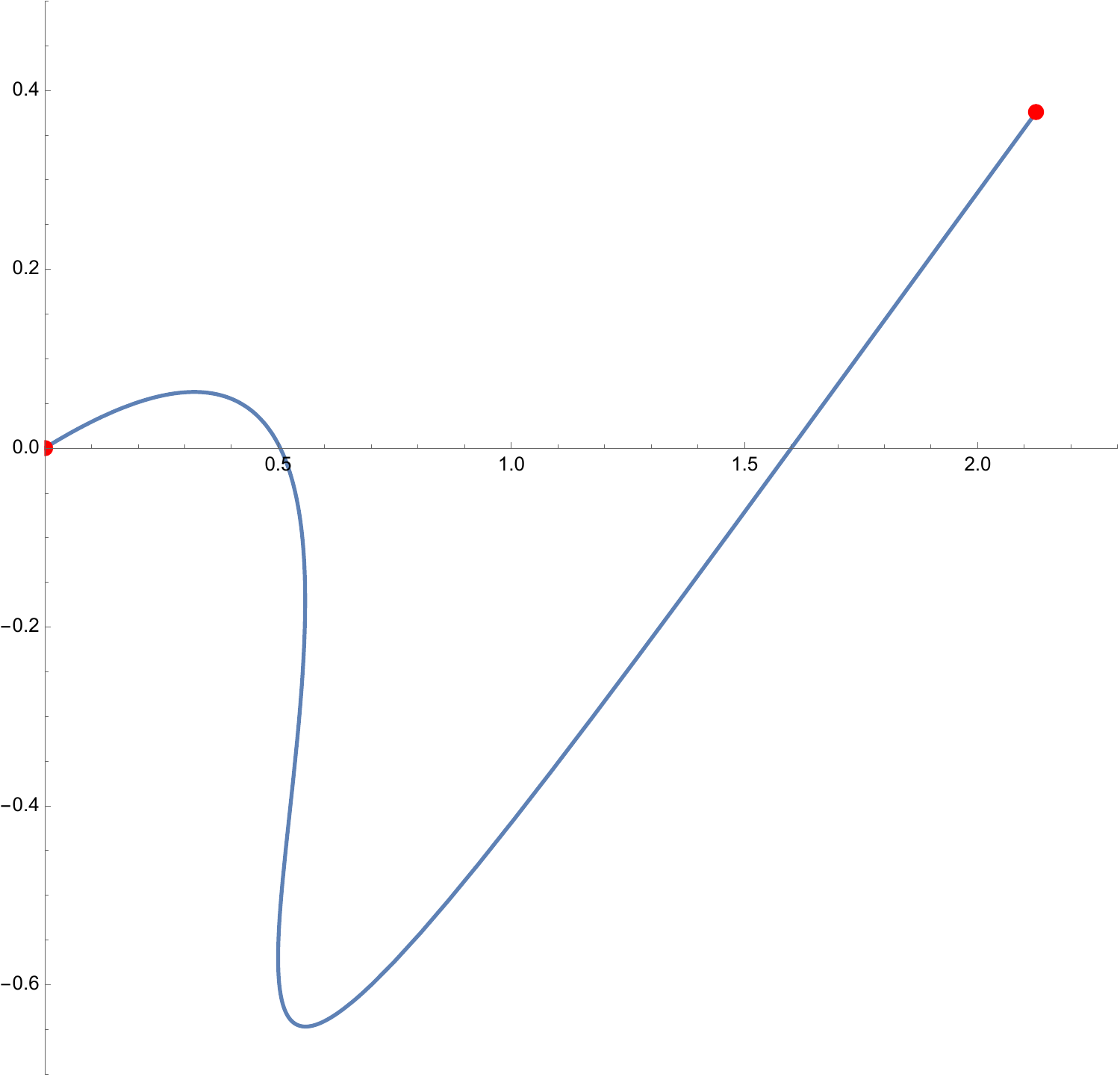}
\caption*{ 
$m=(4)$}
\end{subfigure}
\qquad
    \begin{subfigure}{0.27\columnwidth}
        \centering
\includegraphics[width=\linewidth]{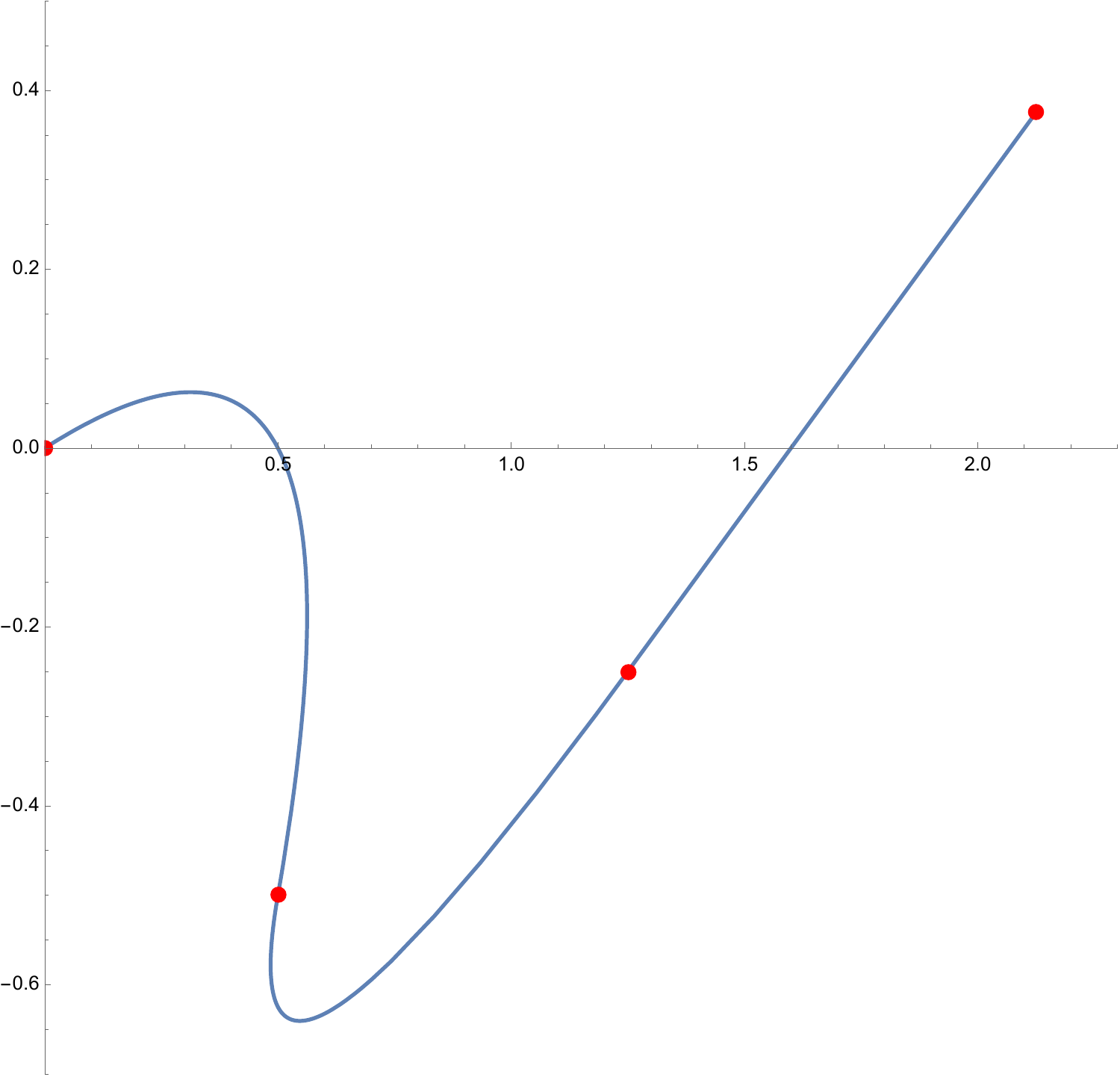}
\caption*{$m=(2,2,1)$}
\end{subfigure}
\hfill

\begin{subfigure}{0.27\columnwidth}
        \centering
\includegraphics[width=\linewidth]{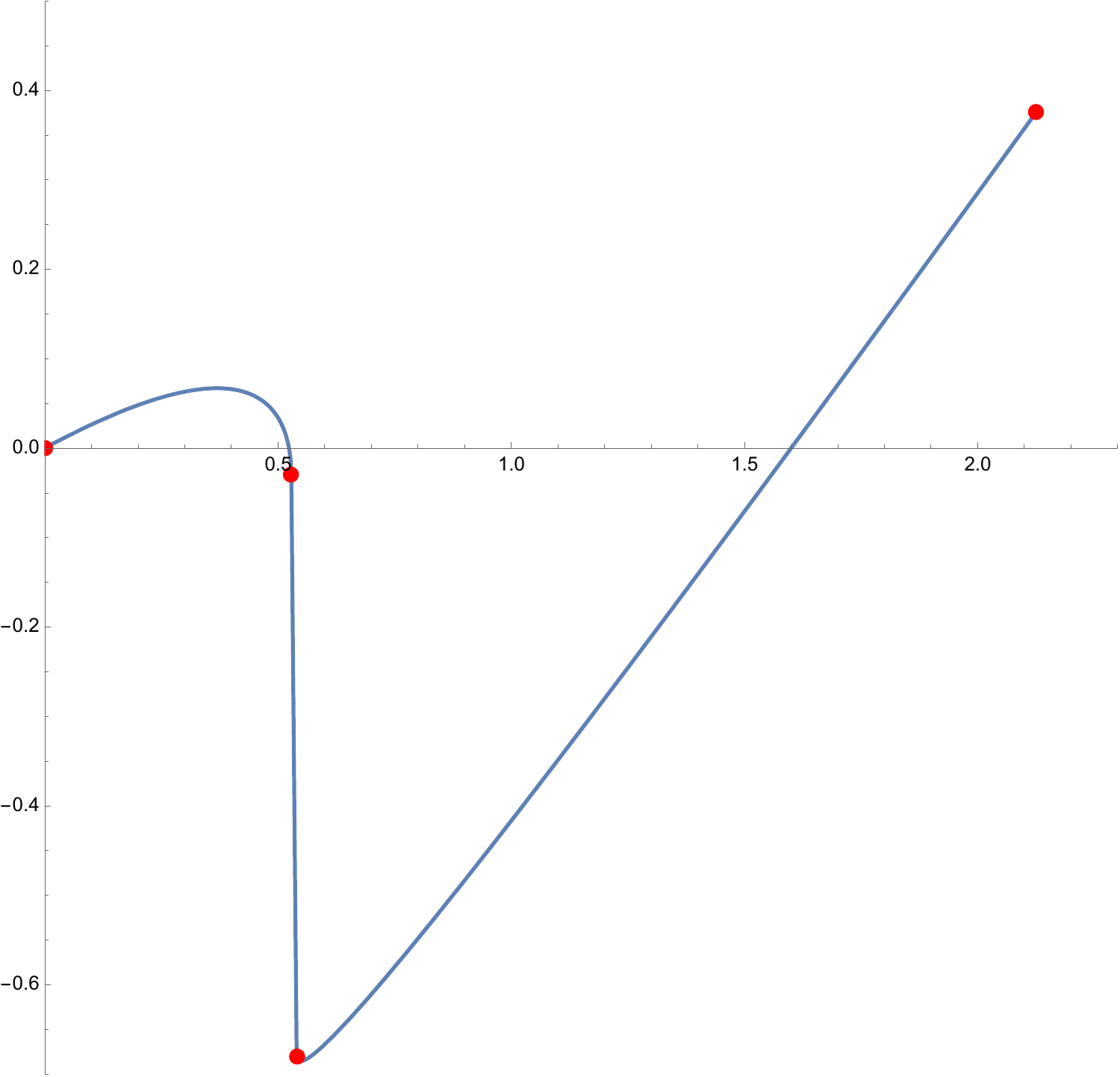}
\caption*{$m=(2,1,2)$}
\end{subfigure}
\qquad
\begin{subfigure}{0.27\columnwidth}
        \centering
\includegraphics[width=\linewidth]{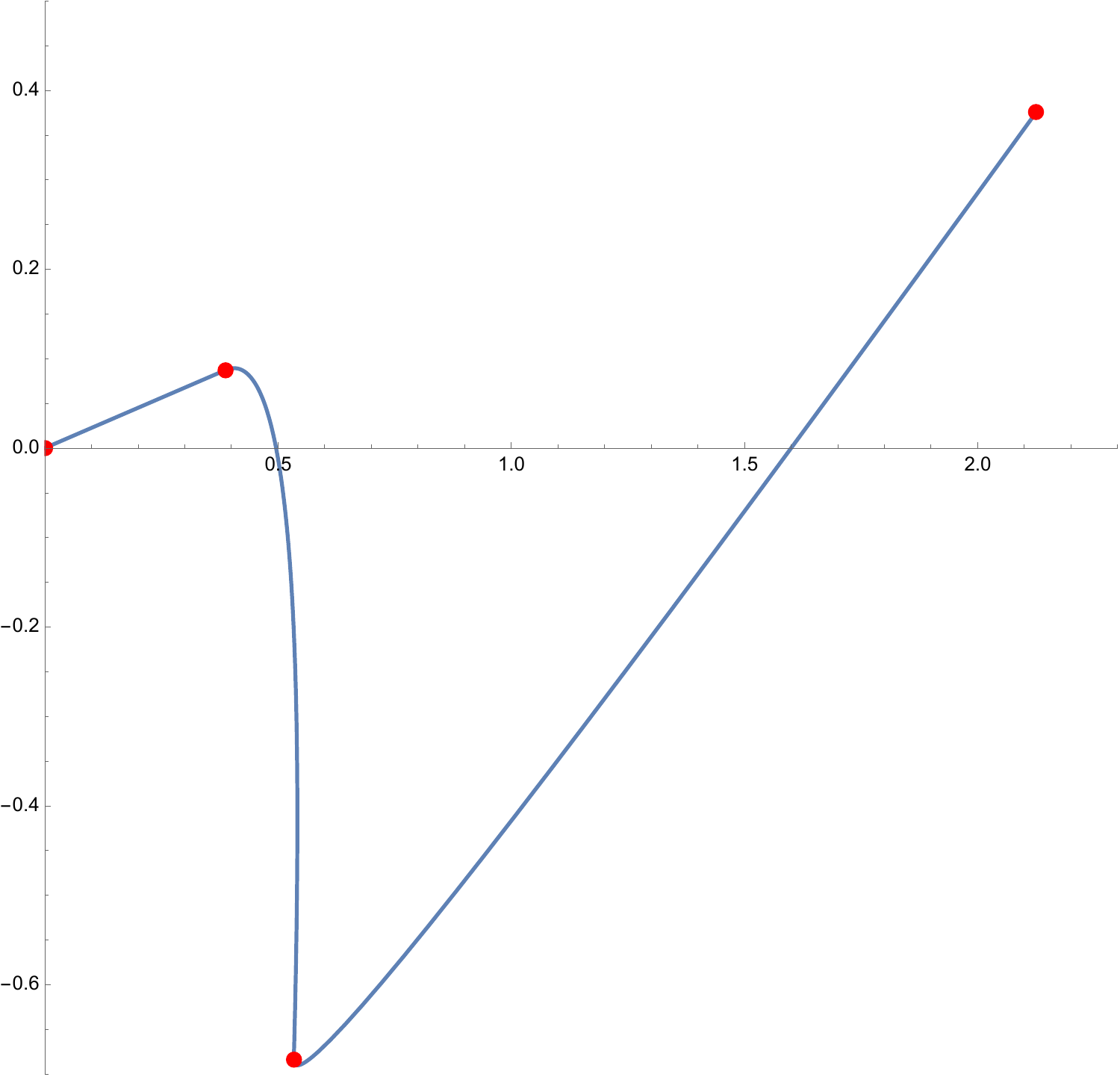}
\caption*{$m=(1,2,2)$}
\end{subfigure}
\caption{Splines of regularity $1$ for different $m$, with the same signature up to level $4$. Control points are marked in red.\label{fig:4-reconstr}}
\end{figure}

Using \texttt{HomotopyContinuation.jl} \cite{BT18} , we computed explicit solutions to a path recovery problem for $\mathcal S^{1}_{2,\leq 4,m}$, for different choices of $m$. According to \Cref{tab:rec-geom-plane}, the expected number of complex solutions is $32$ in each case. Almost all real solutions corresponded to `cuspidal' paths (like the one illustrated in \Cref{fig:m21-fiber-paths}). In fact, the spline solution was always unique, except for the case $m = (4)$, see \Cref{fig:4-reconstr}.

\newpage

\bibliography{refs.bib}
\end{document}